\documentclass[11pt]{article}

\usepackage[letterpaper,margin=1in]{geometry}
\usepackage{amsmath,amssymb,amsthm,mathtools}
\usepackage{stmaryrd,mathrsfs,bm}
\usepackage{booktabs}
\usepackage{microtype}
\usepackage{algorithm}
\usepackage{algpseudocode}
\usepackage{float}
\usepackage{placeins}
\usepackage[authoryear,round]{natbib}
\usepackage[hidelinks]{hyperref}
\usepackage[nameinlink,capitalise]{cleveref}
\usepackage{graphicx}
\usepackage{xcolor}
\hypersetup{colorlinks=false,hidelinks}

\newtheorem{theorem}{Theorem}[section]
\newtheorem{proposition}[theorem]{Proposition}
\newtheorem{lemma}[theorem]{Lemma}
\newtheorem{corollary}[theorem]{Corollary}
\newtheorem{remark}[theorem]{Remark}

\newcommand{\ii}{i}

\newcommand{\C}{\mathbb{C}}
\newcommand{\Th}{\mathcal{T}_h}
\newcommand{\FhI}{\mathcal{F}_h^{\mathrm I}}
\newcommand{\trnorm}[1]{\lVert #1\rVert_{\mathrm{tr}}}
\newcommand{\GR}{\mathrm{GR}}
\newcommand{\PW}{\mathrm{PW}}
\newcommand{\FB}{\mathrm{FB}}
\newcommand{\avg}[1]{\{\!\!\{#1\}\!\!\}}
\newcommand{\jump}[1]{\llbracket #1\rrbracket}
\newcommand{\Jfun}{\mathcal J}
\DeclareMathOperator{\rank}{rank}
\DeclareMathOperator{\spanop}{span}
\DeclareMathOperator{\diag}{diag}

\renewcommand{\Re}{\operatorname{Re}}
\renewcommand{\Im}{\operatorname{Im}}
\newcommand{\DtN}{\operatorname{DtN}}
\numberwithin{equation}{section}
\providecommand{\alttext}[1]{}

\title{Adaptive local representations for Helmholtz Trefftz discontinuous Galerkin methods}
\author{Shelvean Kapita\\[3pt]\small Department of Mathematics, Texas A\&M University, College Station, TX 77843, USA\\\small \href{mailto:kapita@tamu.edu}{kapita@tamu.edu}}
\date{}

\begin{document}
\maketitle

\begin{abstract}
We study the selection and stable realization of local approximation spaces in Trefftz discontinuous Galerkin discretizations of the Helmholtz equation. A scaled Cauchy-trace inner product places plane waves and Fourier--Bessel functions in a common geometry: Fourier--Bessel modes are orthogonal with explicit weights, while the same weights determine the circulant spectrum of an equispaced plane-wave trace Gram matrix. This separates amplitude scaling from genuine trace-rank loss and yields an exact best-approximation identity for mixed plane-wave--Fourier--Bessel spaces. With complex plane-wave angles, the unresolved modal tail is an exponential sequence, so propagating and evanescent components can be identified by the same ESPRIT/variable-projection procedure. We prove exact recovery and a perturbation estimate for the recovered angles and, under the standard PWDG quasi-optimality bound, transfer these perturbations to the DG error. Trace-Riesz orthonormalization is then separated from a graph--Riesz normalization of the assembled operator. Numerical experiments verify the identities, recover sparse ray fields to roundoff, and resolve a propagating-to-evanescent transition without a prescribed critical angle.
\end{abstract}

\vspace{1em}

\section{Introduction}\label{sec:intro}
Trefftz methods use elementwise solutions of the governing differential equation as trial and test functions.  For the homogeneous Helmholtz equation
\begin{equation}\label{eq:helmholtz}
 \Delta u+\kappa^2u=0,
\end{equation}
plane waves give the plane-wave discontinuous Galerkin (PWDG) and ultra-weak variational formulation (UWVF) families; see, for example, \citep{CessenatDespres1998,GittelsonHiptmairPerugia2009,HiptmairMoiolaPerugia2011,HiptmairMoiolaPerugia2016}.  Their approximation properties are particularly attractive for locally directional wave fields.  The practical difficulty is that the local representation itself may become numerically poor before the globally coupled DG problem is assembled.  Crowded plane-wave directions produce strongly correlated traces, whereas high-order Fourier--Bessel (FB) functions may have very small unscaled trace amplitudes.  The first effect is a loss of effective dimension; the second is largely a coordinate-scaling effect.  Previous work has addressed conditioning and alternative local bases \citep{HuttunenMonkKaipio2002,PerreyDebain2006,LuostariHuttunenMonk2012,CongreveGedickePerugia2019,BarucqEtAl2021}, while direction-adaptive approaches optimize or infer dominant PW directions within a prescribed plane-wave family \citep{AmaraEtAl2014,AgrawalHoppe2017,FangEtAl2017,Kapita2026Directions}.  The present problem is different: both the directional PW component and the complementary FB component are selected in one trace metric, and unresolved local rank is removed before any global system is assembled.  The local selection step is independent of whether the retained Trefftz space is subsequently coupled by PWDG or UWVF; the global conditioning analysis below is stated for the PWDG form.

Let $K$ be contained in a disk $B(x_K,h)$ lying in one homogeneous Helmholtz region.  We equip the Cauchy trace on the circle with
\begin{equation}\label{eq:intro-trace}
 \langle v,w\rangle_{\rm tr}
 =\int_{\partial B(x_K,h)}
 \big(\kappa v\overline w+\kappa^{-1}\partial_n v\,\overline{\partial_n w}\big)\,ds.
\end{equation}
For the regular Fourier--Bessel modes $\varphi_m(\rho,\psi)=J_m(\kappa\rho)e^{\ii m\psi}$, \cref{thm:T1} gives
\begin{equation}\label{eq:tau-intro}
 \langle\varphi_m,\varphi_n\rangle_{\rm tr}=\tau_m^2\delta_{mn},\qquad
 \tau_m^2=2\pi h\kappa\big(J_m(\kappa h)^2+J_m'(\kappa h)^2\big),
\end{equation}
and \cref{thm:T2} shows that the same weights determine the eigenvalues of the trace Gram matrix of $q$ equispaced plane waves,
\begin{equation}\label{eq:lambda-intro}
 \lambda_s=q\sum_{m\equiv-s\ ({\rm mod}\ q)}\tau_m^2,
 \qquad s=0,\ldots,q-1.
\end{equation}
Hence small FB amplitudes can be removed by trace equilibration, while small PW Gram eigenvalues identify genuine redundancy.  This distinction is the basis of the local rank test used below.

The second ingredient is an exact representation formula.  For a complex angle $z\in\C$, set
\begin{equation}\label{eq:complex-pw-intro}
 \psi_z(x)=e^{\ii\kappa d(z)\cdot(x-x_K)},\qquad d(z)=(\cos z,\sin z).
\end{equation}
Since $d(z)\cdot d(z)=1$, $\psi_z$ is Trefftz for every $z$.  For a real angle $\theta$, define the orthonormal directions
\[
 t_\theta=(\cos\theta,\sin\theta),\qquad
 t_\theta^\perp=(-\sin\theta,\cos\theta).
\]
If $z=\theta+\ii\eta$, then $d(z)=\cosh(\eta)t_\theta+\ii\sinh(\eta)t_\theta^\perp$, and its oscillatory and exponential factors are
\begin{equation}\label{eq:evanescent-factor}
 \psi_{\theta+\ii\eta}
 =e^{\ii\kappa\cosh\eta\,t_\theta\cdot(x-x_K)}
  e^{-\kappa\sinh\eta\,t_\theta^\perp\cdot(x-x_K)}.
\end{equation}
Thus $\eta=0$ gives propagation in direction $t_\theta$, whereas $\eta\ne0$ introduces exponential variation in the transverse direction $t_\theta^\perp$; propagation and evanescence belong to the same parameterization.  If $u=\sum_m a_m\varphi_m$, $\gamma_m=\ii^{-m}a_m$, and $\mathcal V_{q,M}(Z_q)$ contains $q$ complex-angle plane waves and the FB modes $|m|\le M$, \cref{thm:T3} proves
\begin{equation}\label{eq:hybrid-intro}
 \inf_{v\in\mathcal V_{q,M}(Z_q)}\|u-v\|_{\rm tr}^2
 =\min_{c\in\C^q}\sum_{|m|>M}\tau_m^2
 \left|\gamma_m-\sum_{j=1}^q c_je^{-\ii m z_j}\right|^2.
\end{equation}
After the resolved FB block is removed, selecting the remaining plane waves is therefore a weighted exponential-fitting problem in the modal index.

The main analytical consequences are as follows.  First, \cref{thm:T1,thm:T2,thm:T3} provide a common local approximation geometry for FB, PW and mixed spaces.  Second, the exponential structure in \eqref{eq:hybrid-intro} permits target-frequency direction identification by ESPRIT, followed by a small variable-projection problem.  \Cref{thm:esprit-exact,thm:esprit-stability} give exact recovery for finite exponential sums and perturbation stability of the recovered complex angles.  Third, \cref{lem:pw-direction-lipschitz,thm:esprit-pwdg-bridge} connect these local angle perturbations to element, boundary-trace and PWDG errors; the final DG estimate is conditional only on the standard quasi-optimality bound for the fixed PWDG flux configuration.  Fourth, stability is treated at two distinct levels.  The selected local family is orthonormalized in the Cauchy-trace metric, whereas the assembled matrix is normalized in the DG graph metric.  \Cref{prop:GR} shows that the latter transformation produces the normal matrix $S-\ii I$ with $S$ Hermitian.  Local trace rank and global operator conditioning are therefore not identified with one another.

The disk construction is a local analytical device, not a geometric restriction on the mesh.  When a homogeneous continuation to a containing circle is unavailable---for example near a source, material interface, obstacle boundary or re-entrant corner---the same candidate spaces are compared by Cauchy least squares directly on the physical element boundary.  The numerical experiments test both settings.  They verify the trace identities independently, compare sparse and distributed angular content, exhibit the high-order rank ceiling of plane-wave traces, recover a propagating-to-evanescent transmission transition without supplying the critical angle, and compare local trace selection with a globally coupled residual search.

The paper is organized as follows.  \Cref{sec:dg} states the Trefftz-DG formulation and graph metric.  \Cref{sec:trace-theory} develops the local trace identities and hybrid approximation formula.  \Cref{sec:conditioning} separates effective local dimension from the conditioning of the assembled system.  \Cref{sec:optimization} develops local direction identification, variable projection and the perturbation-to-DG estimate, and then records the globally coupled residual formulation used for comparison.  \Cref{sec:numerics} presents the numerical results.  Conclusions are given in \cref{sec:conclusion}.

\section{Trefftz-DG formulation and graph metric}\label{sec:dg}
\subsection{PWDG formulation and graph metric}
We use $e^{-\ii\omega t}$, so $e^{\ii\kappa d\cdot x}$ propagates in direction $d$, outgoing cylindrical waves use $H_m^{(1)}$, and
\begin{equation}\label{eq:outgoing}
 \partial_nu-\ii\kappa u=g_A.
\end{equation}
Set the dimensionless flux
\begin{equation}\label{eq:sigma-def}
 \sigma=-\frac{1}{\ii\kappa}\nabla u,\qquad
 \nabla u+\ii\kappa\sigma=0,\qquad
 \nabla\cdot\sigma+\ii\kappa u=0.
\end{equation}
Let $\Th$ be a shape-regular mesh with interior faces $\FhI$.  For $F=\partial K^+\cap\partial K^-$ with outward normals $n^\pm$,
\begin{equation}\label{eq:jumps}
 \avg{v}=\tfrac12(v^++v^-),\quad
 \jump{v}_N=v^+n^++v^-n^-,\quad
 \jump{\nabla v}_N=\nabla v^+\cdot n^+ + \nabla v^-\cdot n^- ,
\end{equation}
where the subscript $N$ labels a normal jump; in particular, $\jump{\nabla v}_N$ is a scalar normal-flux jump and not a norm of $\nabla v$.  The broken Trefftz space is $T(\Th)=\{v\in H^1(\Th):\ \Delta v+\kappa^2v=0\ \text{in every }K\in\Th\}$.

On interior faces we use the conjugated PWDG flux family \citep{CessenatDespres1998,GittelsonHiptmairPerugia2009,HiptmairMoiolaPerugia2011}
\begin{equation}\label{eq:interior-fluxes}
 \widehat u=\avg{u}+\frac{\beta}{\ii\kappa}\jump{\nabla u}_N,
 \qquad
 \widehat\sigma=\avg{\sigma}-\alpha\jump{u}_N,\qquad \alpha,\beta>0.
\end{equation}
On a Dirichlet boundary $\Gamma_D$ with $u=g_D$ we take $\widehat u=g_D$ and $\widehat\sigma=\sigma-\alpha(u-g_D)n$.  On an impedance boundary $\Gamma_A$ satisfying \eqref{eq:outgoing}, with $r_A(u)=\partial_nu-\ii\kappa u-g_A$,
\begin{equation}\label{eq:imp-flux}
 \widehat u=u+\frac{\delta}{\ii\kappa}r_A(u),
 \qquad
 \widehat\sigma=\sigma-(1-\delta)\Big(\sigma\cdot n+u+\frac{g_A}{\ii\kappa}\Big)n,
 \qquad 0<\delta<1.
\end{equation}
Unless stated otherwise, all computations use the symmetric choice $\alpha=\beta=\delta=1/2$.  The transmission tests have no impedance boundary and use $\alpha=\beta=1/2$ with the interface scaling stated in \cref{sec:numerics}.
For Trefftz trial and test functions, elementwise Green identities leave only skeleton terms \citep{HiptmairMoiolaPerugia2011,HiptmairMoiolaPerugiaSurvey}.  The discrete problem is: find $u_h\in V_h\subset T(\Th)$ with
\begin{equation}\label{eq:dg-problem}
 \mathscr A_h(u_h,v_h)=\ell_h(v_h)\qquad\forall v_h\in V_h,
\end{equation}
\begin{align}
\mathscr A_h(u,v)
={}&\int_{\FhI}\Big(
 \avg{u}\,\overline{\jump{\nabla v}_N}
 -\avg{\nabla u}\cdot\overline{\jump{v}_N}\Big)ds
 \nonumber\\
 &-\int_{\FhI}\Big(
 \alpha\ii\kappa\jump{u}_N\cdot\overline{\jump{v}_N}
 -\frac{\beta}{\ii\kappa}\jump{\nabla u}_N\,\overline{\jump{\nabla v}_N}\Big)ds
 \nonumber\\
 &+\int_{\Gamma_D}\Big(-\partial_nu\,\overline v-\alpha\ii\kappa u\overline v\Big)ds
 \nonumber\\
 &+\int_{\Gamma_A}\Big(
 -\delta(\partial_nu\,\overline v+u\,\overline{\partial_nv})
 -\ii(1-\delta)\kappa u\overline v
 +\frac{\delta}{\ii\kappa}\partial_nu\,\overline{\partial_nv}\Big)ds,
 \label{eq:Ah-interior}\\
\ell_h(v)
={}&-\int_{\Gamma_D}g_D\big(\overline{\partial_nv}+\alpha\ii\kappa\overline v\big)ds
\nonumber\\
 &+\int_{\Gamma_A}\Big((1-\delta)g_A\overline v+\frac{\delta}{\ii\kappa}g_A\overline{\partial_nv}\Big)ds.
 \label{eq:lh}
\end{align}
With this convention the penalty terms have negative imaginary part because $1/\ii=-\ii$.  For $v\in T(\Th)$ one obtains the graph identity, the analogue of the coercive skeleton norm of PWDG analysis \citep{GittelsonHiptmairPerugia2009,HiptmairMoiolaPerugia2011},
\begin{align}\label{eq:graph-identity}
 -\Im\mathscr A_h(v,v)
={}&\alpha\kappa\|\jump v_N\|_{0,\FhI}^2
 +\beta\kappa^{-1}\|\jump{\nabla v}_N\|_{0,\FhI}^2
 +\alpha\kappa\|v\|_{0,\Gamma_D}^2\nonumber\\
 &+(1-\delta)\kappa\|v\|_{0,\Gamma_A}^2
 +\delta\kappa^{-1}\|\partial_nv\|_{0,\Gamma_A}^2.
\end{align}
Thus, if $K_h$ is the stiffness matrix of \eqref{eq:dg-problem}, the positive graph matrix is
\begin{equation}\label{eq:G-correct-sign}
 G_h:=\frac{K_h^*-K_h}{2\ii}>0.
\end{equation}
The sign is convention dependent; under $e^{+\ii\omega t}$ the opposite sign is positive.

\begin{proposition}[Definiteness of the PWDG graph metric]\label{prop:graph-definite}
Assume $\alpha,\beta>0$ and $0<\delta<1$.  If $\Gamma_A$ contains a nonempty relatively open boundary segment, then the right-hand side of \eqref{eq:graph-identity} is a norm on every finite-dimensional Trefftz space $V_h$.  If $\Gamma_A=\varnothing$, the same conclusion holds provided $\kappa^2$ is not a Dirichlet eigenvalue of $-\Delta$ on $\Omega$.  Consequently, in either case $G_h>0$ for every linearly independent coefficient basis of $V_h$.
\end{proposition}
\begin{proof}
If the graph seminorm of $v\in V_h$ vanishes, both the value and normal-flux jumps vanish on every interior face.  Hence the broken Trefftz field is a global Helmholtz solution with continuous Cauchy data across the mesh.  On $\Gamma_D$ its trace is zero.  If $\Gamma_A\ne\varnothing$, \eqref{eq:graph-identity} also gives $v=\partial_n v=0$ on an open boundary segment, so Cauchy uniqueness gives $v=0$.  If $\Gamma_A=\varnothing$, the homogeneous Dirichlet problem has only the zero solution by the nonresonance hypothesis.  Positivity of the Gram matrix follows.
\end{proof}

\begin{remark}[Interior resonance]\label{rem:resonance}
When $\Gamma_A=\varnothing$, the exclusion of Dirichlet eigenfrequencies is a property of the underlying boundary-value problem, not of the local representation.  Away from exact resonance $G_h$ is positive, but the continuous resolvent and the corresponding quasi-optimality constants can deteriorate as $\kappa^2$ approaches the Dirichlet spectrum.  Representation selection and graph--Riesz normalization do not remove this loss of problem stability.  For exterior-scattering computations the impedance or DtN boundary terms used here avoid this pure-Dirichlet situation; for an interior problem resonance must be treated at the PDE level, for example by changing the boundary condition or by using a shifted formulation appropriate to the application.
\end{remark}

For the truncated nonlocal DtN form we use the corresponding coercivity result of \citet{KapitaMonk2018}; in floating point the Cholesky test remains an implementation check on the retained coordinates, not the source of continuous stability.

\subsection{DtN truncation and graph--Riesz coordinates}\label{sec:dtn-form}
The scattering experiments use a circular Dirichlet-to-Neumann (DtN) boundary $r=R$.  A first-order absorbing condition would add reflection error to the same traces used to judge the local representation, while the Fourier--Hankel DtN map removes that ambiguity, so changes in the error can be attributed to the local representation alone.  Circular DtN maps are classical exact nonreflecting boundary conditions \citep{KellerGivoli1989}, follow from separation of variables \citep{ColtonKressBook}, and were used with PWDG by \citet{KapitaMonk2018}.  The truncated map and residual are
\begin{equation}\label{eq:dtn-map}
 \DtN_Nv=\sum_{|m|\le N}\kappa\frac{H_m^{(1)\prime}(\kappa R)}{H_m^{(1)}(\kappa R)}v_me^{\ii m\theta},
 \qquad r_N(u)=\partial_nu-\DtN_Nu-g_R,
\end{equation}
with the nonlocal flux pair
\begin{equation}\label{eq:dtn-fluxes}
 \widehat u=u+\frac{\delta}{\ii\kappa}r_N(u),
 \qquad
 \ii\kappa\,\widehat\sigma\cdot n
 =-(\DtN_Nu+g_R)-\frac{\delta}{\ii\kappa}\DtN_N^*r_N(u).
\end{equation}
For homogeneous data the boundary form is
\begin{equation}\label{eq:dtn-form}
\mathscr A_R(u,v)=\int_{\Gamma_R}\Big[
 u\,\overline{\partial_nv}-(\DtN_Nu)\overline v
 +\frac{\delta}{\ii\kappa}(\partial_nu-\DtN_Nu)\overline{(\partial_nv-\DtN_Nv)}\Big]ds,
\end{equation}
the $H^{(1)}$ conjugate of the form in \citet{KapitaMonk2018}.  In the computations we form $G_h$ from \eqref{eq:G-correct-sign} and verify definiteness numerically.

After the local space has been fixed, the assembled operator is normalized in the graph metric.
Let $K_h\in\C^{n_h\times n_h}$ be the PWDG matrix after local selection and trace orthonormalization, and write
\begin{equation}\label{eq:Kh-HG-split}
 H_h:=\frac{K_h+K_h^*}{2},\qquad
 G_h:=\frac{K_h^*-K_h}{2\ii},\qquad
 K_h=H_h-\ii G_h.
\end{equation}
For the stable fluxes used here $G_h=G_h^*>0$ on the retained space.  Let $B$ satisfy
\begin{equation}\label{eq:GR-factor}
 B^*G_hB=I.
\end{equation}
The implementation uses $G_h=LL^*$ and $B=L^{-*}$; failure of the Cholesky factorization makes the retained space inadmissible.  Any two exact factors satisfying \eqref{eq:GR-factor} differ by a unitary matrix.

\begin{proposition}[Graph--Riesz normal form]\label{prop:GR}
Let
\begin{equation}\label{eq:Khat-def}
 \widehat K_h:=B^*K_hB,\qquad S:=B^*H_hB.
\end{equation}
Then $\widehat K_h=S-\ii I$ with $S=S^*$.  Hence $\widehat K_h$ is normal,
\begin{equation}\label{eq:GR-singular-values}
 \sigma_j(\widehat K_h)=\sqrt{1+\lambda_j(S)^2},\qquad
 \sigma_{\min}(\widehat K_h)\ge1.
\end{equation}
\end{proposition}
\begin{proof}
$S=B^*H_hB$ is Hermitian and \eqref{eq:GR-factor} gives
$B^*K_hB=B^*H_hB-\ii B^*G_hB=S-\ii I$.  A Hermitian matrix commutes with the identity, so $S-\ii I$ is normal.  Its singular values are the moduli of the eigenvalues $\lambda_j(S)-\ii$, which gives \eqref{eq:GR-singular-values}.
\end{proof}

In exact arithmetic, \eqref{eq:GR-factor} gives
\begin{equation}\label{eq:kappa-GR-def}
 \kappa_{\GR}:=\kappa_2(\widehat K_h)
 =\frac{\max_j\sqrt{1+\lambda_j(S)^2}}
        {\min_j\sqrt{1+\lambda_j(S)^2}}.
\end{equation}
Thus $\sigma_{\min}(\widehat K_h)\ge1$ and growth of $\kappa_{\GR}$ is governed by the Hermitian part $S$.  Numerical values of $\kappa_{\GR}$ are computed by SVD of the actual floating-point matrix $B^*K_hB$.  Local rank decisions use only the trace metric; $\kappa_{\GR}$ measures the final globally coupled solve.

\section{Local Cauchy-trace geometry}\label{sec:trace-theory}
\subsection{Trace metric and modal identities}
Fix $K$ and a disk $B_K=B(x_K,h)\supset K$ contained in the same homogeneous Helmholtz region.  This local analyticity assumption is the only role of the disk; near sources, interfaces, or obstacle boundaries the selector instead uses element-boundary trace fitting.  Regular Helmholtz fields on $B_K$ admit the Fourier--Bessel expansion \citep{Watson1944,MoiolaHiptmairPerugia2011}.
Write
\begin{equation}\label{eq:radius-admissible}
 r_K:=\sup_{x\in K}|x-x_K|,\qquad
 d_K:=\operatorname{dist}(x_K,\Sigma_K),
\end{equation}
where $\Sigma_K$ is the nearest source, material interface, obstacle boundary, or other set across which the same regular Helmholtz continuation is unavailable.  The disk construction is admissible whenever
\begin{equation}\label{eq:radius-window}
 r_K<h<d_K.
\end{equation}
All identities below hold for every such $h$.  Numerically, $h$ also fixes $\kappa h$ and hence the modal weights $\tau_m$; it is therefore part of the rank-resolution scale, not a physical mesh parameter.  We use the smallest convenient containing disk with a modest geometric margin.  If \eqref{eq:radius-window} is empty, no disk continuation is invoked and the same candidate spaces are compared by Cauchy least squares on $\partial K$.
\begin{equation}\label{eq:FBexpansion}
 u(\rho,\psi)=\sum_{m\in\mathbb Z}a_m\varphi_m(\rho,\psi),\qquad
 \varphi_m=J_m(\kappa\rho)e^{\ii m\psi},
\end{equation}
and on $\partial B_K$ we define the scaled Cauchy trace inner product
\begin{equation}\label{eq:trace-inner-disk}
 \langle v,w\rangle_{\mathrm{tr}}
 =\int_{\partial B_K}\big(\kappa v\overline w+\kappa^{-1}\partial_nv\,\overline{\partial_nw}\big)ds.
\end{equation}
The weights are the value/flux scaling of the PWDG graph norm.

The Fourier--Bessel basis diagonalizes this metric.
\begin{theorem}[Trace orthogonality]\label{thm:T1}
The functions $\varphi_m$ are orthogonal in \eqref{eq:trace-inner-disk}:
\begin{equation}\label{eq:T1}
 \langle\varphi_m,\varphi_n\rangle_{\mathrm{tr}}=\tau_m^2\delta_{mn},\qquad
 \tau_m^2=2\pi h\kappa\big(J_m(\kappa h)^2+J_m'(\kappa h)^2\big),
\end{equation}
and $\tau_m>0$ for every $m$ when $\kappa h>0$.
\end{theorem}
\begin{proof}
On $\rho=h$, $\varphi_m=J_m(\kappa h)e^{\ii m\psi}$ and $\partial_n\varphi_m=\kappa J_m'(\kappa h)e^{\ii m\psi}$.  Substitute into \eqref{eq:trace-inner-disk} with $ds=h\,d\psi$ and use Fourier orthogonality.  If $J_m$ and $J_m'$ vanished at the same $z>0$, uniqueness for the Bessel equation with zero Cauchy data would give $J_m\equiv0$.
\end{proof}

\begin{corollary}[Equilibrated modal isometry]\label{cor:isometry}
If $u$ has the expansion \eqref{eq:FBexpansion}, then
\begin{equation}\label{eq:isometry}
 \trnorm{u}^2=\sum_{m\in\mathbb Z}|a_m|^2\tau_m^2,
\end{equation}
and the quantities
\begin{equation}\label{eq:bm}
 b_m:=\tau_ma_m=\big\langle u,\varphi_m/\tau_m\big\rangle_{\mathrm{tr}}
\end{equation}
are orthonormal coordinates of the trace.  The map $u\mapsto(b_m)$ is an isometry into $\ell^2$.
\end{corollary}
\begin{remark}\label{rem:raw-coeff}
The isometry concerns the equilibrated coefficients $b_m$, not the raw coefficients $a_m=b_m/\tau_m$.  At high order $\tau_m$ is very small (\cref{fig:tau}a), so recovering unscaled $a_m$ amplifies roundoff.  Rank decisions must be made after trace equilibration.
\end{remark}

\begin{figure}[H]
\centering
\includegraphics[width=.72\textwidth]{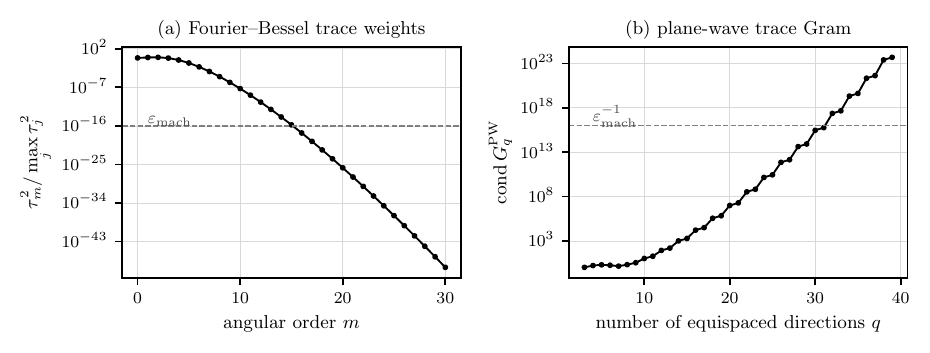}
\caption{Local trace geometry at $k=16$ and $kh=3.34$. Left: Fourier--Bessel trace weights $\tau_m$. Right: condition number of the equispaced plane-wave trace Gram matrix obtained from \eqref{eq:T2-eigs}.}
\alttext{Alt text: Two-panel plot of local trace geometry at $k=16$. Left: the Fourier--Bessel trace weights decay rapidly with mode number. Right: the condition number of the equispaced plane-wave trace Gram matrix increases sharply as the number of plane-wave directions grows.}
\label{fig:tau}
\end{figure}

The same weights also determine the spectrum of equispaced propagating plane-wave traces.  For $z\in\C$ let $\psi_z$ be \eqref{eq:complex-pw-intro}.  Jacobi--Anger gives, first for real $z$ and then for complex $z$ by analytic continuation,
\begin{equation}\label{eq:JA}
 \psi_z(\rho,\psi)=\sum_{m\in\mathbb Z}\ii^mJ_m(\kappa\rho)e^{\ii m\psi}e^{-\ii mz}.
\end{equation}
Thus the phase-corrected modal sequence of one complex-angle plane wave is $e^{-\ii mz}$.

\begin{theorem}[Plane-wave Gram symbol]\label{thm:T2}
For real angles $\theta_j,\theta_\ell$,
\begin{equation}\label{eq:T2-Gram}
 \langle\psi_{\theta_j},\psi_{\theta_\ell}\rangle_{\mathrm{tr}}
 =\sum_{m\in\mathbb Z}\tau_m^2e^{-\ii m(\theta_j-\theta_\ell)}.
\end{equation}
If $\theta_j=2\pi j/q$, $j=0,\ldots,q-1$, then $G_q^{\PW}$ is circulant with
\begin{equation}\label{eq:T2-eigs}
 \lambda_s=q\sum_{m\equiv -s\pmod q}\tau_m^2,
 \qquad s=0,\ldots,q-1.
\end{equation}
\end{theorem}
\begin{proof}
Insert \eqref{eq:JA} into \eqref{eq:trace-inner-disk} and use \cref{thm:T1}.  For equispaced real angles, a discrete Fourier transform leaves precisely the residue classes in \eqref{eq:T2-eigs}.
\end{proof}
Small PW eigenvalues are therefore correlation, not amplitude scaling; diagonal equilibration cannot remove them.

\subsection{Hybrid approximation and representation crossover}
For $Z_q=(z_1,\ldots,z_q)\in\C^q$ define
\begin{equation}\label{eq:hybrid-space-new}
 \mathcal V_{q,M}(Z_q)=\spanop\{\psi_{z_j}:1\le j\le q\}
 +\spanop\{\varphi_m:|m|\le M\}.
\end{equation}

\begin{theorem}[Hybrid error]\label{thm:T3}
Let $u$ satisfy \eqref{eq:FBexpansion}, $\gamma_m=\ii^{-m}a_m$, and $|\Im z_j|\le\eta_{\max}<\infty$.  Then
\begin{equation}\label{eq:T3}
 \inf_{v\in\mathcal V_{q,M}(Z_q)}\trnorm{u-v}^2
 =\min_{c\in\C^q}\sum_{|m|>M}\tau_m^2
 \left|\gamma_m-\sum_{j=1}^qc_je^{-\ii mz_j}\right|^2.
\end{equation}
Repeated angles are allowed; they only make the PW coefficient vector nonunique.
\end{theorem}
\begin{proof}
By \eqref{eq:JA}, the PW block contributes $\sum_jc_je^{-\ii mz_j}$ after removal of the phase $\ii^m$.  The bounded-strip hypothesis and the factorial decay of $J_m(\kappa h)$ and $J_m'(\kappa h)$ imply $\tau_m e^{|m|\eta_{\max}}\in\ell^2$, so the weighted modal series is well defined.  For fixed $c$, the FB coefficients with $|m|\le M$ cancel the low-order residual exactly.  Apply \cref{cor:isometry} to the remaining tail and minimize over $c$.
\end{proof}

\begin{corollary}[Sparse complex rays]\label{cor:sparse}
If $u$ is a superposition of $q_*$ plane waves $\psi_{z_j}$ whose distinct complex angles are contained in $Z_q$, then the right side of \eqref{eq:T3} vanishes for $q\ge q_*$, independently of $M$.
\end{corollary}
\begin{proof}
Choose the generating PW coefficients and set the others to zero.  Equation \eqref{eq:JA} reproduces every modal coefficient.
\end{proof}

The hybrid identity also yields an exact field-dependent crossover criterion.
Let
\begin{equation}\label{eq:rho-field}
 u(\rho)=u_{\rm ray}+\rho u_{\rm diff},\qquad \rho\ge0,
\end{equation}
where $u_{\rm ray}$ is a finite ray field and $u_{\rm diff}$ has nonsparse equilibrated modal content on the resolved window.  Let $g_{\rm ray},g_{\rm diff}\in\ell^2$ be their equilibrated coefficient vectors from \cref{cor:isometry}; hence
\begin{equation}\label{eq:g-rho}
 g(\rho)=g_{\rm ray}+\rho g_{\rm diff}.
\end{equation}
For a fixed candidate space $V$, let $P_V$ be the orthogonal projector onto its equilibrated modal image.  Then
\begin{equation}\label{eq:EV}
 E_V(\rho)^2:=\inf_{v\in V}\trnorm{u(\rho)-v}^2
 =\|(I-P_V)g(\rho)\|_2^2
 =A_V+2\rho B_V+\rho^2C_V,
\end{equation}
where, with $r_{\rm ray}=(I-P_V)g_{\rm ray}$ and $r_{\rm diff}=(I-P_V)g_{\rm diff}$,
\begin{equation}\label{eq:ABC}
 A_V=\|r_{\rm ray}\|_2^2,\qquad
 B_V=\Re(r_{\rm ray}^*r_{\rm diff}),\qquad
 C_V=\|r_{\rm diff}\|_2^2.
\end{equation}

For two candidate spaces $V_1$ and $V_2$, equality of the two errors is equivalent to
\begin{equation}\label{eq:crossover-poly}
 (C_{V_1}-C_{V_2})\rho^2+2(B_{V_1}-B_{V_2})\rho+(A_{V_1}-A_{V_2})=0 .
\end{equation}
Thus the representation crossover is obtained from a scalar quadratic.  If \eqref{eq:crossover-poly} has no nonnegative real root, the ordering of $V_1$ and $V_2$ is fixed for all $\rho\ge0$; otherwise its nonnegative roots partition that half-line into intervals of fixed ordering.  No DG solve is required for this comparison.

\section{Stable local coordinates and effective trace dimension}\label{sec:conditioning}
For a selected local family $\Psi_K=[\psi_1,\ldots,\psi_{p_K}]$, define
\begin{equation}\label{eq:Gtr-raw-def}
 G_{K,\rm raw}^{\rm tr}=\big(\langle\psi_j,\psi_\ell\rangle_{{\rm tr},K}\big)_{j,\ell},
 \qquad
 D_K=\diag\big(\|\psi_j\|_{{\rm tr},K}^{-1}\big),
 \qquad
 G_{K,\rm eq}^{\rm tr}=D_K^*G_{K,\rm raw}^{\rm tr}D_K.
\end{equation}
The diagonal scaling removes trace amplitude; the spectrum of $G_{K,\rm eq}^{\rm tr}$ measures correlation.  For a Hermitian positive semidefinite Gram matrix $G$, set
\begin{equation}\label{eq:reff}
 r_{\rm eff}(G;\varepsilon)=\#\{j:\lambda_j(G)\ge\varepsilon\lambda_{\max}(G)\}.
\end{equation}
PW directions are retained only when the equilibrated PW block has full nominal rank.  FB modes are equilibrated before any rank decision; the complete hybrid block is then tested for cross-family redundancy.  If $T_K$ is the resulting trace-Riesz map,
\begin{equation}\label{eq:Gtr-orth-def}
 T_K^*G_{K,\rm raw}^{\rm tr}T_K=I,
 \qquad
 \kappa_{\rm tr,final}:=\max_K\kappa_2(T_K^*G_{K,\rm raw}^{\rm tr}T_K).
\end{equation}
This is a local coordinate quantity.  The assembled PWDG operator has the separate graph--Riesz condition number $\kappa_{\GR}$ of \eqref{eq:kappa-GR-def}.

\begin{table}[htbp]
\centering
\caption{Local trace-Gram condition numbers at $\kappa h=3.34$.  PW columns are unscaled; the FB column is trace-equilibrated.}
\label{tab:localcond}
\footnotesize
\begin{tabular}{rrrr}
\toprule
local count & PW Gram, disk & PW Gram, element & equilibrated FB, element\\
\midrule
9  & $3.6$ & $1.7\times10^1$ & $1.95$\\
15 & $1.9\times10^3$ & $4.8\times10^4$ & $9.9$\\
21 & $2.0\times10^7$ & $2.6\times10^9$ & $7.0\times10^1$\\
27 & $1.4\times10^{12}$ & $8.5\times10^{14}$ & $5.0\times10^2$\\
33 & $>10^{15}$ & rank deficient & $4.2\times10^3$\\
\bottomrule
\end{tabular}
\end{table}

The selected local space is orthonormalized in the trace metric before global assembly.
Local rank loss is removed before global coupling \citep{HuttunenMonkKaipio2002,CongreveGedickePerugia2019,BarucqEtAl2021}.

Scale the selected family by $D_K$.  Apply the rank threshold to the equilibrated PW block, retain resolvable FB modes, then apply the same threshold to the full hybrid Gram to remove cross-family redundancy.  If $Q_r,\Lambda_r$ are the retained eigendata of the resulting raw trace Gram, define
\begin{equation}\label{eq:local-trace-riesz}
 \Phi_K=\Psi_KT_K,\qquad
 T_K=Q_r\Lambda_r^{-1/2},\qquad
 T_K^*G_{K,\rm raw}^{\rm tr}T_K=I.
\end{equation}
All subsequent assembly uses the retained trace-orthonormal space.  In floating point we report the condition number of the recomputed matrix $T_K^*G_{K,\rm raw}^{\rm tr}T_K$.  Thus $T_K$ is obtained explicitly from the retained eigenpairs of the local Gram matrix, not from an additional optimization problem.  If $p_K$ local functions are sampled at $Q_K$ trace quadrature points, forming the Gram matrix costs $O(Q_Kp_K^2)$ and its Hermitian eigendecomposition costs $O(p_K^3)$; both operations are element local and are performed once before global assembly.

\begin{algorithm}[t]
\caption{Local conditioning and global graph--Riesz solve}\label{alg:conditioned-solve}
\begin{algorithmic}[1]
\For{each element $K$}
  \State select the family $\Psi_K$ (PW, FB, or mixed) by \cref{alg:selector};
  \State scale every basis function by its local trace norm;
  \State diagonalize the equilibrated PW trace Gram matrix and discard eigenvectors with $\lambda<\tau_{\rm rank}\lambda_{\max}$;
  \State retain the equilibrated FB block unless it is numerically null;
  \State diagonalize the retained full hybrid trace Gram matrix, remove any cross-family eigenvectors below the same rank tolerance, and store the trace-Riesz map $T_K$.
\EndFor
\State assemble the Trefftz-DG system in the block-diagonal conditioned coordinates;
\State form $G_h=(K_h^*-K_h)/(2\ii)$, verify $\lambda_{\min}(G_h)>0$, compute a Cholesky factor $G_h=LL^*$, and set $B=L^{-*}$;
\State solve $(B^*K_hB)\widehat c=B^*f$ and back-transform.
\end{algorithmic}
\end{algorithm}

The transformations act on different objects: $T_K$ fixes the local representation; $B$ normalizes the assembled PWDG operator.

\Cref{tab:conditioned-pipeline} separates the local and global condition numbers.  The PW rows lose redundant directions before assembly; the FB rows retain their equilibrated dimensions.

\begin{table}[htbp]
\centering
\caption{Local and global conditioning at $\kappa=8$ on eight curved sectors.  $\kappa_{\rm tr,raw}$ is the worst raw local trace-Gram condition number; $\kappa_{\rm tr,final}$ is measured after trace orthonormalization; $\kappa_{\GR}$ belongs to the assembled graph--Riesz PWDG matrix.}
\label{tab:conditioned-pipeline}
\footnotesize
\begin{tabular}{rcrrrr}
\toprule
$p$ & $(q_{\PW},q_{\FB})$ & retained & $\kappa_{\rm tr,raw}$ & $\kappa_{\rm tr,final}$ & $\kappa_{\GR}$\\
\midrule
15 & $(15,0)$ & 104 & $9.6\times10^{13}$ & $1.000001$ & 3.12\\
15 & $(0,15)$ & 120 & $4.9\times10^{4}$ & $1.000000$ & 3.67\\
15 & $(7,8)$  & 120 & $2.8\times10^{10}$ & $1.000002$ & 3.63\\
27 & $(27,0)$ & 120 & $1.2\times10^{18}$ & $1.000001$ & 4.00\\
27 & $(0,27)$ & 216 & $8.9\times10^{14}$ & $1.000000$ & 7.93\\
27 & $(11,16)$& 168 & $1.4\times10^{19}$ & $1.000006$ & 5.37\\
\bottomrule
\end{tabular}
\end{table}

For any recovered PW candidate $Z_q$, admissibility means
\begin{equation}\label{eq:rank-admissible}
 r_{\rm eff}(G_q^{\PW}(Z_q);\varepsilon_{\rm solve})=q.
\end{equation}
For equispaced real angles the eigenvalues are given explicitly by \cref{thm:T2}.  The parameter $\varepsilon_{\rm solve}$ denotes the smallest relative trace scale that the subsequent numerical solve is intended to resolve, and $\tau_{\rm rank}$ is the relative eigenvalue threshold used in the local Gram matrices.  They are numerical-accuracy parameters, not physical parameters and not universal constants.  For clarity, let $\varepsilon_{\rm floor}:=\max\{\varepsilon_{\rm mach},\varepsilon_{\rm lin},\varepsilon_{\rm data}\}$ denote an estimated relative numerical floor, where $\varepsilon_{\rm lin}$ is the attained relative linear-algebra accuracy and $\varepsilon_{\rm data}$ is the relative uncertainty of the trace data.  A practical choice should satisfy $\varepsilon_{\rm solve}\gtrsim\varepsilon_{\rm floor}$.  The reported binary64 direct-solve experiments use $\tau_{\rm rank}=\varepsilon_{\rm solve}=10^{-12}$ unless a sensitivity sweep states otherwise.  The wavenumber and trace radius enter indirectly through the Gram spectrum and the resolvable modal window rather than through a separate prescribed scaling of $\varepsilon_{\rm solve}$.

\section{Direction identification and local optimization}\label{sec:optimization}
The hybrid identity \eqref{eq:hybrid-intro} turns local PW selection into a small exponential-fitting problem.  We first develop this local procedure and its stability, and only afterwards record the globally coupled residual formulation used as a cost comparison in \cref{sec:numerics}.

\subsection{Local modal variable projection}
The local nonlinear problem is the hybrid identity itself.  For unresolved modes $|m|>M$, define
\begin{equation}\label{eq:modal-matrices}
 W(Z)_{m,j}=e^{-\ii mz_j},\qquad A(Z)=D_\tau W(Z),\qquad y=D_\tau\gamma,
\end{equation}
where $D_\tau=\diag(\tau_m)$.  Then
\begin{equation}\label{eq:modal-varpro}
 F_{q,M}(Z,\bar Z)=\min_{c\in\C^q}\|y-A(Z)c\|_2^2.
\end{equation}
On a constant-rank stratum let $c=A^\dagger y$ and $r=y-Ac$.  The envelope theorem and $A_{\bar z_j}=0$ give the Wirtinger derivative
\begin{equation}\label{eq:modal-wirtinger}
 \partial_{\bar z_j}F_{q,M}
 =-\overline{c_j}\,(\partial_{z_j}a_j)^*r,
 \qquad
 (\partial_{z_j}a_j)_m=-\ii m\tau_m e^{-\ii mz_j},
\end{equation}
where $a_j$ is the $j$th column of $A$.  Coefficients are computed by QR or SVD, not by forming normal equations.  Thus each trial uses only small element-local dense linear algebra; the admissible strip $|\Im z_j|\le\eta_{\max}$ prevents arbitrarily growing evanescent PWs.

\subsection{Stability-aware local selection}\label{sec:selector}
For each element the selector returns an admissible $(q,M,Z_q)$.  It acts on equilibrated modal data and enforces trace rank before comparing approximation errors.

\paragraph{Modal recovery}
By \cref{cor:isometry} the stable quantities are $b_m=\langle u,\varphi_m/\tau_m\rangle_{\mathrm{tr}}$.  We use a consecutive modal window
\[
 \mathcal W=\{m_0,m_0+1,\ldots,m_0+N-1\}
\]
on which the conversion from equilibrated coefficients to $\gamma_m$ is resolved: $\tau_m$ is large enough that the estimated perturbation of $b_m/\tau_m$ remains below the requested modal accuracy.  The phase-corrected sequence $\gamma_m=\ii^{-m}b_m/\tau_m$ is analyzed for exponential structure.  Reindexing by $\widetilde\gamma_r=\gamma_{m_0+r}$ changes only the exponential coefficients, not the nodes $\zeta_j=e^{-\ii z_j}$, because $\zeta_j^{m_0}$ is absorbed into the corresponding coefficient.

On a containing circle the quantities $b_m$ are Fourier coefficients of the scaled Cauchy data.  With $Q$ trace samples, direct projection onto $N$ retained modes costs $O(QN)$ per element; with equispaced samples the Fourier part can be evaluated simultaneously by an FFT in $O(Q\log Q)$.  This trace extraction is performed once per selection update, not at every variable-projection trial.  When the containing-circle hypothesis fails, no modal inversion is attempted: the candidate spaces are compared directly by Cauchy least squares on the physical element boundary.

\paragraph{ESPRIT for sparse exponential content}
For a finite complex-ray field $u_{\rm ray}=\sum_{j=1}^{q_*}c_j\psi_{z_j}$, \eqref{eq:JA} gives
\begin{equation}\label{eq:esprit-sequence}
 \gamma_m=\sum_{j=1}^{q_*}c_j\zeta_j^m,
 \qquad \zeta_j=e^{-\ii z_j}\ne0.
\end{equation}
Thus propagation ($\Im z_j=0$) corresponds to $|\zeta_j|=1$, whereas an evanescent PW has $|\zeta_j|\ne1$.  For the reindexed window $\widetilde\gamma_0,\ldots,\widetilde\gamma_{N-1}$ we use, unless stated otherwise, the balanced dimensions
\[
 L=\left\lceil\frac{N+1}{2}\right\rceil,\qquad K=N-L+1,
\]
so that $L+K-1=N$.  A requested rank $q$ is admissible only when $L\ge q+1$ and $K\ge q$.  We form
\begin{equation}\label{eq:esprit-H}
 H=(\widetilde\gamma_{r+s})_{r=0,\ldots,L-1}^{s=0,\ldots,K-1}.
\end{equation}
The dominant $q$-dimensional left singular subspace is retained only when its $q$th singular value is above the corresponding data-accuracy floor.  If $U\in\C^{L\times q}$ spans that subspace and $U_0,U_1$ delete its last and first rows, respectively, set
\begin{equation}\label{eq:esprit-S}
 S=U_0^\dagger U_1.
\end{equation}

\begin{theorem}[Exact target-frequency recovery]\label{thm:esprit-exact}
Assume \eqref{eq:esprit-sequence} has exactly $q$ nonzero coefficients and pairwise distinct nodes $\zeta_j$.  If $L\ge q+1$ and $K\ge q$, then $\rank H=q$, $U_0$ has full column rank, and
\begin{equation}\label{eq:esprit-similarity}
 S=T^{-1}\diag(\zeta_1,\ldots,\zeta_q)T
\end{equation}
for a nonsingular $T$.  Hence ESPRIT recovers every $\zeta_j$ exactly and the complex angles follow from
\begin{equation}\label{eq:esprit-angle-recovery}
 z_j=\ii\operatorname{Log}\zeta_j\pmod{2\pi},
\end{equation}
with a fixed logarithm branch.
\end{theorem}
\begin{proof}
Let $(V_L)_{rj}=\zeta_j^r$, $(V_K)_{sj}=\zeta_j^s$, and $C=\diag(c_j)$.  Then
\begin{equation}\label{eq:esprit-factorization}
 H=V_LCV_K^T.
\end{equation}
Distinct nodes and nonzero coefficients give $\rank H=q$ and $U=V_LT$ for some nonsingular $T$.  If $V_0,V_1$ are the shifted row blocks of $V_L$, then
\begin{equation}\label{eq:vandermonde-shift}
 V_1=V_0\diag(\zeta_1,\ldots,\zeta_q).
\end{equation}
Since $V_0$ has full column rank, multiplication by $U_0^\dagger$ yields \eqref{eq:esprit-similarity}.
\end{proof}

\begin{theorem}[Perturbation of ESPRIT nodes and complex angles]\label{thm:esprit-stability}
Let $\widetilde H=H+E$ and construct $\widetilde S$ from the dominant $q$-dimensional left singular subspace of $\widetilde H$.  Put
\[
 \sigma_*:=\sigma_q(H)>0,
 \qquad s_*:=\sigma_q(U_0)>0.
\]
For sufficiently small $\|E\|_2$, the eigenvalues $\widetilde\zeta_j$ of $\widetilde S$ can be labeled so that
\begin{equation}\label{eq:esprit-stability-bound}
 \max_j|\widetilde\zeta_j-\zeta_j|\le C_{\rm E}\|E\|_2,
\end{equation}
where $C_{\rm E}$ depends on $\sigma_*^{-1}$, $s_*^{-1}$, and $\kappa_2(T)$.  If the exact nodes lie in a compact annulus that avoids the chosen logarithm cut, then
\begin{equation}\label{eq:esprit-angle-stability}
 \max_j|\widetilde z_j-z_j|
 \le C_{\log}C_{\rm E}\|E\|_2,
 \qquad \widetilde z_j=\ii\operatorname{Log}\widetilde\zeta_j.
\end{equation}
\end{theorem}
\begin{proof}
Wedin's singular-subspace bound gives $\|\widetilde UQ-U\|_2\le C\|E\|_2/\sigma_*$ for a unitary $Q$.  Since $s_*>0$, the shifted pseudoinverse is locally Lipschitz, hence $\|Q^*\widetilde SQ-S\|_2\le C'\|E\|_2$.  Bauer--Fike applied to \eqref{eq:esprit-similarity} gives \eqref{eq:esprit-stability-bound}.  The logarithm is locally Lipschitz on any compact set separated from zero and its branch cut, which gives \eqref{eq:esprit-angle-stability}.
\end{proof}

\begin{corollary}[Consistency inside the hybrid selector]\label{cor:esprit-hybrid}
If
\begin{equation}\label{eq:esprit-hybrid-pert}
 \gamma_m^{(M)}=\sum_{j=1}^qc_j\zeta_j^m+r_m^{(M)}
\end{equation}
with distinct nonzero $\zeta_j$, and the Hankel matrix generated by $r_m^{(M)}$ tends to zero in spectral norm, then the recovered nodes and complex angles converge, up to permutation, to the exact ones.  The same conclusion holds with an additional vanishing trace-discretization perturbation.
\end{corollary}
\begin{proof}
Apply \cref{thm:esprit-stability} to the total Hankel perturbation.
\end{proof}

\begin{lemma}[Complex-angle Lipschitz bounds]\label{lem:pw-direction-lipschitz}
Let $K\subset B(x_K,h_K)$ and $|\Im z|,|\Im w|\le\eta_{\max}$.  Define
\[
 D_\eta=\sqrt{\cosh(2\eta_{\max})},
 \qquad E_K=e^{\kappa h_K\sinh\eta_{\max}},
\]
and introduce the augmented element-boundary norm
\begin{equation}\label{eq:boundary-plus-norm}
 \|v\|_{\partial K,+}^2
 :=\int_{\partial K}\big(\kappa|v|^2+\kappa^{-1}|\nabla v|^2\big)\,ds.
\end{equation}
Then
\begin{align}
 \|\psi_z^K-\psi_w^K\|_{H^1(K)}
 &\le C_{H^1,K}|z-w|,\label{eq:pw-H1-lipschitz}\\
 \|\psi_z^K-\psi_w^K\|_{\mathrm{tr},K}
 &\le C_{\mathrm{tr},K}|z-w|,\label{eq:pw-tr-lipschitz}\\
 \|\psi_z^K-\psi_w^K\|_{\partial K,+}
 &\le C_{\partial,K}|z-w|,\label{eq:pw-boundary-lipschitz}
\end{align}
where one may take
\begin{align}
 C_{H^1,K}
 &=|K|^{1/2}\kappa D_\eta E_K
 \Big(h_K^2+(1+\kappa h_KD_\eta)^2\Big)^{1/2},\label{eq:CH1K}\\
 C_{\mathrm{tr},K}
 &=\Big(2\pi h_K\kappa D_\eta^2E_K^2
 [ (\kappa h_K)^2+(1+\kappa h_KD_\eta)^2]\Big)^{1/2},\label{eq:CtrK}\\
 C_{\partial,K}
 &=\Big(|\partial K|\,\kappa D_\eta^2E_K^2
 [ (\kappa h_K)^2+(1+\kappa h_KD_\eta)^2]\Big)^{1/2}.\label{eq:CboundaryK}
\end{align}
For $\eta_{\max}=0$ these reduce to the corresponding real-angle bounds.  The factor $E_K=e^{\kappa h_K\sinh\eta_{\max}}$ is intrinsic: strongly evanescent directions are exponentially amplified when continued across a patch of radius comparable with $h_K$.  This is the reason for imposing a bounded complex-angle strip rather than allowing arbitrarily large $|\Im z|$.
\end{lemma}
\begin{proof}
Along the segment joining $z$ and $w$,
\[
 \partial_z\psi_z=\ii\kappa d'(z)\cdot(x-x_K)\psi_z,
 \qquad
 \partial_z\nabla\psi_z=\ii\kappa d'(z)\psi_z+\ii\kappa d(z)\partial_z\psi_z.
\]
On the strip, $\|d(z)\|_2,\|d'(z)\|_2\le D_\eta$ and $|\psi_z|\le E_K$.  Hence, uniformly for $x\in B(x_K,h_K)$,
\[
 |\partial_z\psi_z|\le \kappa D_\eta h_KE_K,
 \qquad
 |\partial_z\nabla\psi_z|\le \kappa D_\eta E_K(1+\kappa h_KD_\eta).
\]
Integration along the segment in the complex-angle plane gives the pointwise difference bounds.  Integration over $K$, over $\partial B(x_K,h_K)$, and over $\partial K$, respectively, yields \eqref{eq:pw-H1-lipschitz}--\eqref{eq:pw-boundary-lipschitz}.  On the circle, $|\partial_n v|\le|\nabla v|$, which gives \eqref{eq:CtrK}.
\end{proof}

For the bridge estimate we use the standard PWDG continuity norm
\begin{align}\label{eq:DGplus-def}
 \|v\|_{\mathrm{DG}+}^2
 :={}&-\Im\mathscr A_h(v,v)
 +\kappa^{-1}\|\avg{\nabla v}\|_{0,\FhI}^2
 +\kappa\|\avg v\|_{0,\FhI}^2
 +\kappa^{-1}\|\partial_n v\|_{0,\Gamma_D}^2,
\end{align}
with the finite-dimensional DtN trace terms added on $\Gamma_R$ when \eqref{eq:dtn-form} is used.  The face inequalities for jumps and averages imply that, for piecewise Trefftz $v$,
\begin{equation}\label{eq:DGplus-boundary-control}
 \|v\|_{\mathrm{DG}+}^2
 \le C_{\rm face}\sum_{K\in\Th}\|v\|_{\partial K,+}^2,
\end{equation}
where $C_{\rm face}$ depends only on the fixed flux parameters and face multiplicity; for a fixed truncated DtN map it also contains its finite-dimensional trace-operator bound.  This is the only mesh-level constant needed below.

\begin{theorem}[From modal perturbation to PWDG error]\label{thm:esprit-pwdg-bridge}
Suppose
\begin{equation}\label{eq:local-finite-ray-bridge}
 u|_K=\sum_{j=1}^{q_K}c_{Kj}\psi_{z_{Kj}}^K,
 \qquad |\Im z_{Kj}|\le\eta_{\max},
\end{equation}
and let the recovered angles $\widehat z_{Kj}$ lie in the same strip.  Set
\[
 \delta_K^2=\sum_{j=1}^{q_K}|\widehat z_{Kj}-z_{Kj}|^2.
\]
If $V_h(\widehat Z)$ contains the corresponding recovered plane waves, then
\begin{align}
 \inf_{v_h\in V_h(\widehat Z)}
 \Big(\sum_K\|u-v_h\|_{H^1(K)}^2\Big)^{1/2}
 &\le\Big(\sum_K C_{H^1,K}^2\|c_K\|_2^2\delta_K^2\Big)^{1/2},\label{eq:bridge-H1}\\
 \inf_{v_h\in V_h(\widehat Z)}
 \Big(\sum_K\|u-v_h\|_{\mathrm{tr},K}^2\Big)^{1/2}
 &\le\Big(\sum_K C_{\mathrm{tr},K}^2\|c_K\|_2^2\delta_K^2\Big)^{1/2}.\label{eq:bridge-tr}
\end{align}
If the standard PWDG quasi-optimality estimate
\begin{equation}\label{eq:pwdg-qo-assumption}
 \|u-u_h\|_{\mathrm{DG}}
 \le C_{\mathrm{qo}}\inf_{v_h\in V_h(\widehat Z)}\|u-v_h\|_{\mathrm{DG}+}
\end{equation}
holds for the fixed flux configuration, then
\begin{equation}\label{eq:bridge-DG-angle}
 \|u-u_h(\widehat Z)\|_{\mathrm{DG}}
 \le C_{\mathrm{qo}}C_{\rm face}^{1/2}
 \Big(\sum_K C_{\partial,K}^2\|c_K\|_2^2\delta_K^2\Big)^{1/2}.
\end{equation}
\end{theorem}
\begin{proof}
Use the comparison function $v_K=\sum_jc_{Kj}\psi_{\widehat z_{Kj}}^K$.  Applying \cref{lem:pw-direction-lipschitz} to each summand and then Cauchy--Schwarz in the ray index gives \eqref{eq:bridge-H1} and \eqref{eq:bridge-tr}; the same argument with \eqref{eq:pw-boundary-lipschitz} yields
\[
 \sum_K\|u-v_h\|_{\partial K,+}^2
 \le \sum_K C_{\partial,K}^2\|c_K\|_2^2\delta_K^2.
\]
Combine this inequality with \eqref{eq:DGplus-boundary-control} and \eqref{eq:pwdg-qo-assumption} to obtain \eqref{eq:bridge-DG-angle}.
\end{proof}

The estimate inherits the stability constant $C_{\rm qo}$ of the underlying PWDG formulation.  No uniformity of this constant is asserted near interior resonances, on strongly under-resolved meshes, or in coefficient-contrast regimes outside the hypotheses of the corresponding PWDG stability theory; local direction recovery cannot compensate for loss of stability of the continuous or discrete boundary-value problem.

Combining \eqref{eq:bridge-DG-angle} with \cref{thm:esprit-stability} gives the data-to-DG estimate
\begin{equation}\label{eq:bridge-DG-hankel}
 \|u-u_h(\widehat Z)\|_{\mathrm{DG}}
 \le C_{\mathrm{qo}}C_{\rm face}^{1/2}
 \Big(\sum_K q_K C_{\partial,K}^2\|c_K\|_2^2
 (C_{\log,K}C_{{\rm E},K}\|E_K\|_2)^2\Big)^{1/2}.
\end{equation}

\begin{remark}[General fields]\label{rem:bridge-general}
If $u=u_{\rm ray}+r$ and $u_{\rm ray}$ has the form \eqref{eq:local-finite-ray-bridge}, the same argument gives
\[
 \|u-u_h(\widehat Z)\|_{\mathrm{DG}}
 \le C_{\mathrm{qo}}\left(E_{\rm best}
 +C_{\rm face}^{1/2}
 \Big(\sum_K C_{\partial,K}^2\|c_K\|_2^2\delta_K^2\Big)^{1/2}\right),
\]
where $E_{\rm best}$ is the best $\mathrm{DG}+$ approximation of the nondirectional remainder.  The finite-ray case has $E_{\rm best}=0$.
\end{remark}

\paragraph{Stability filter and score}
Approximation is compared only after the recovered candidate is shown to be numerically realizable.  For a candidate $Z_q$ define
\begin{equation}\label{eq:admissibleq}
 \mathcal A_q(Z_q):\quad
 \max_j|\Im z_j|\le\eta_{\max},\qquad
 r_{\rm eff}(G_q^{\PW}(Z_q);\varepsilon_{\rm solve})=q.
\end{equation}
The explicit equispaced spectrum \eqref{eq:T2-eigs} remains a cheap capacity estimate for real PW sweeps; the actual selector uses the trace Gram matrix of the recovered real or complex angles.

\begin{algorithm}[H]
\caption{Stability-aware local PW--FB selection}\label{alg:selector}
\begin{algorithmic}[1]
\State compute equilibrated modal data $b_m$ on a resolved window and fix a local budget $p$;
\State set $\mathcal Q(p)=\{p\}\cup\{q:0\le q<p,\ p-q=2M+1,\ M\in\mathbb N_0\}$;
\For{$q\in\mathcal Q(p)$}
 \State if $q>0$, recover complex angles $Z_q$ by ESPRIT and polish \eqref{eq:modal-varpro} subject to $|\Im z_j|\le\eta_{\max}$;
 \State reject $Z_q$ unless \eqref{eq:admissibleq} holds;
 \State if $q=p$, evaluate the pure-PW trace error; otherwise set $M=(p-q-1)/2$ and form the full equilibrated PW--FB trace Gram matrix;
 \State reject a mixed candidate of effective rank $<p$; otherwise evaluate \eqref{eq:T3};
\EndFor
\State return the admissible $(q,M,Z_q)$ of minimum trace error.
\end{algorithmic}
\end{algorithm}

\subsection{Practical resolution and parameter choice}\label{sec:practical-resolution}
The selector uses only information that is numerically resolved.  For a candidate PW rank $q$, the modal window must be long enough for the shifted ESPRIT system and must remain above the trace-data floor.  With the balanced choice above, an odd window of length $N$ permits at most $q\le(N-1)/2$.  If no resolved window is long enough for a proposed $q$, that PW candidate is rejected rather than recovered from coefficients dominated by roundoff.  At high frequency the number of significant Fourier modes on a patch grows on the scale of $\kappa h$; consequently the modal-extraction and Hankel costs also grow.  Sparse direction recovery is therefore most advantageous when the number of dominant directions is small compared with this angular bandwidth.  For broad angular content, the trace-error comparison naturally favors a larger FB component, as illustrated in \cref{sec:broad-content}.

The rank test and approximation test have different roles.  A large $\kappa_H=\sigma_1(H)/\sigma_q(H)$ indicates sensitivity of an assumed rank-$q$ exponential model, but a small $\kappa_H$ does not imply that the data are close to such a model.  Structured nondirectional content can perturb the recovered angles even when the retained signal subspace is well conditioned.  Such model mismatch is measured by the residual in \eqref{eq:modal-varpro} and ultimately by the trace error in \eqref{eq:T3}; a candidate is accepted only after both the numerical-rank test and the approximation comparison.  Thus singular values determine whether a rank can be resolved, whereas the trace residual determines whether that rank is an adequate representation.

The containing-disk and physical-boundary routes also use different data.  When \eqref{eq:radius-window} holds, the orthogonal modal coordinates provide the PW--FB comparison developed in \cref{sec:trace-theory}.  If it does not hold, the method falls back to weighted Cauchy least squares on $\partial K$ and does not invoke the disk identities or divide by unresolved Bessel factors.  This fallback remains a local comparison of candidate spaces; it is not claimed to recover an analytic continuation that does not exist.

\subsection{Comparison with globally coupled residual optimization}\label{sec:global-residual}
On a fixed-rank stratum let $z=(z_1,\ldots,z_m)\in\C^m$ collect the active complex plane-wave angles.  Discrete PW--FB allocations and trace-rank changes are outer events.  For fixed $z$ the PWDG state is
\begin{equation}\label{eq:global-state}
 K(z,\bar z)c=f(z,\bar z),
\end{equation}
and the weighted skeleton residual has the quadratic form
\begin{equation}\label{eq:Jresidual-matrix}
 \Jfun(c,z,\bar z)=c^*R(z,\bar z)c-c^*b(z,\bar z)-b(z,\bar z)^*c+\gamma(z,\bar z),
 \qquad R=R^*\ge0.
\end{equation}
This is the matrix form of the jump and boundary residual used in adaptive PWDG \citep{KapitaMonkWarburton2015,Kapita2026Directions}; $R$ is a residual Gram matrix and is unrelated to the Galerkin matrix $K$.

Write $z_j=\theta_j+\ii\eta_j$ and
\[
 \partial_{z_j}=\tfrac12(\partial_{\theta_j}-\ii\partial_{\eta_j}),
 \qquad
 \partial_{\bar z_j}=\tfrac12(\partial_{\theta_j}+\ii\partial_{\eta_j}).
\]
For the holomorphic local plane wave $\psi_{z_j}(x)=\exp(\ii\kappa d(z_j)\cdot(x-x_K))$, $K_{z_j}$ differentiates its trial column and $K_{\bar z_j}$ the conjugated test row.  The coefficient gradient of \eqref{eq:Jresidual-matrix} is
\begin{equation}\label{eq:wirtinger-c}
 g_c:=\partial_{\bar c}\Jfun=Rc-b.
\end{equation}
The DG-constrained residual problem is
\begin{equation}\label{eq:constrained-opt}
 \Phi_{\rm DG}(z,\bar z)=\Jfun(c(z,\bar z),z,\bar z),
 \qquad Kc=f.
\end{equation}
One adjoint solve,
\begin{equation}\label{eq:wirtinger-adjoint}
 K^*\lambda=g_c,
\end{equation}
gives the complete reduced Wirtinger derivative.  Indeed, differentiating \eqref{eq:global-state} with respect to $z_j$ and $\bar z_j$ and eliminating the two coefficient sensitivities yields
\begin{align}\label{eq:wirtinger-global-gradient}
 \partial_{\bar z_j}\Phi_{\rm DG}
 ={}&c^*R_{\bar z_j}c-c^*b_{\bar z_j}-b_{z_j}^*c+\gamma_{\bar z_j}\nonumber\\
 &+(f_{z_j}-K_{z_j}c)^*\lambda
   +\lambda^*(f_{\bar z_j}-K_{\bar z_j}c).
\end{align}
For the usual conjugated-test load, $f_{z_j}=0$.  Since $\Phi_{\rm DG}$ is real, first-order stationarity is simply
\begin{equation}\label{eq:wirtinger-stationarity}
 \partial_{\bar z}\Phi_{\rm DG}=0.
\end{equation}

If the Galerkin constraint is dropped, fixed-$z$ stationarity is simply $Rc_{\rm LS}=b$.  On a constant-rank stratum,
\begin{equation}\label{eq:varpro}
 c_{\rm LS}=R^\dagger b,
 \qquad
 \Phi_{\rm LS}=\gamma-b^*R^\dagger b,
\end{equation}
and the Wirtinger envelope formula is
\begin{equation}\label{eq:wirtinger-ls-gradient}
 \partial_{\bar z_j}\Phi_{\rm LS}
 =c_{\rm LS}^*R_{\bar z_j}c_{\rm LS}
  -c_{\rm LS}^*b_{\bar z_j}-b_{z_j}^*c_{\rm LS}+\gamma_{\bar z_j}.
\end{equation}
The coefficient derivative disappears because $\partial_{\bar c}\Jfun=0$.  This is the complex variable-projection formula \citep{GolubPereyra1973,GolubPereyra2003}.  The constrained and unconstrained residual formulations therefore differ only in whether the PWDG state equation is enforced; both are differentiated entirely in Wirtinger coordinates.

The local selector above operates directly at the target frequency and does not require frequency continuation once the modal sequence is resolved.  Continuation is used only for the separate globally coupled residual optimization in \eqref{eq:constrained-opt}.

\section{Numerical results}\label{sec:numerics}
All computations use binary64 arithmetic.  Unless stated otherwise, $\tau_{\rm rank}=10^{-12}$ and every PWDG solve uses local trace orthonormalization followed by graph--Riesz normalization.  Complex-angle optimizations use the bounded strip $|\Im z|\le\eta_{\max}$ with $\eta_{\max}=1.8$ in the reported experiments.

For the transmission test we reproduce the fluid--fluid interface model of \citet{MascottoPichler2020}.  The lower and upper half-squares have $\kappa_j=\omega n_j$.  For incidence angle $\theta_i$ measured from the horizontal interface, set
\[
 \xi=\kappa_1\cos\theta_i,
 \qquad \eta_1=\kappa_1\sin\theta_i,
 \qquad \eta_2=(\kappa_2^2-\xi^2)^{1/2},
 \quad \Re\eta_2,\Im\eta_2\ge0.
\]
The exact field is
\begin{equation}\label{eq:transmission}
 u=\begin{cases}
 e^{\ii(\xi x+\eta_1y)}+R e^{\ii(\xi x-\eta_1y)},&y<0,\\
 T e^{\ii(\xi x+\eta_2y)},&y>0,
 \end{cases}
 \qquad
 R=\dfrac{\eta_1-\eta_2}{\eta_1+\eta_2},\quad
 T=\dfrac{2\eta_1}{\eta_1+\eta_2}.
\end{equation}
The transmitted complex angle is
\begin{equation}\label{eq:transmitted-complex-angle}
 z_t(\theta_i)=
 \begin{cases}
 \arccos(\xi/\kappa_2),&\xi\le\kappa_2,\\
 \ii\operatorname{arcosh}(\xi/\kappa_2),&\xi>\kappa_2.
 \end{cases}
\end{equation}
For $n_1=2$, $n_2=1$, $\omega=12$, the critical angle is $\theta_c=\arccos(\kappa_2/\kappa_1)=60^\circ$.  Equation \eqref{eq:transmitted-complex-angle} is continuous through $z_t=0$: above $60^\circ$ it is real, below $60^\circ$ it is purely imaginary.

Representative exact fields are shown in \cref{fig:solution-gallery}.  The first is the regular Fourier--Bessel solution $J_5(16r)e^{5\ii\theta}$; the other two are \eqref{eq:transmission} on opposite sides of the critical angle.  Each panel is a true 2D projection shown with equal $x$--$y$ scaling, the physical domain boundary, and a visible colormap.  Color encodes $\Re u$ after panelwise normalization.  The normalization is used only for visualization; all errors below are computed from the unscaled complex fields.  In the $29^\circ$ case the transmitted component loses its oscillatory normal wavenumber and decays exponentially into the upper medium.

\begin{figure}[H]
\centering
\begin{minipage}[b]{.29\textwidth}
\centering
\includegraphics[width=\linewidth]{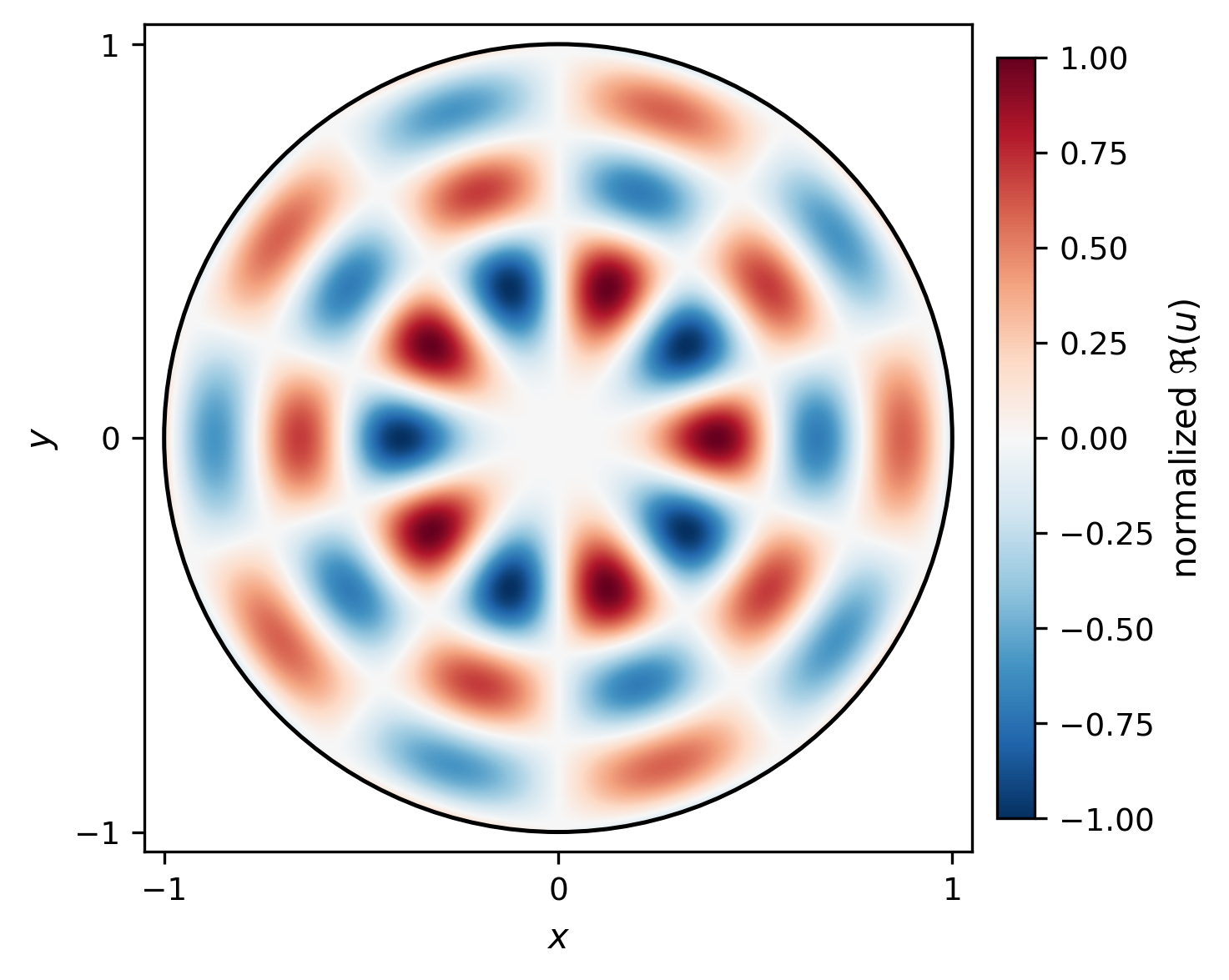}\\[-1mm]
{\footnotesize (a) Fourier--Bessel solution}
\end{minipage}\hfill
\begin{minipage}[b]{.29\textwidth}
\centering
\includegraphics[width=\linewidth]{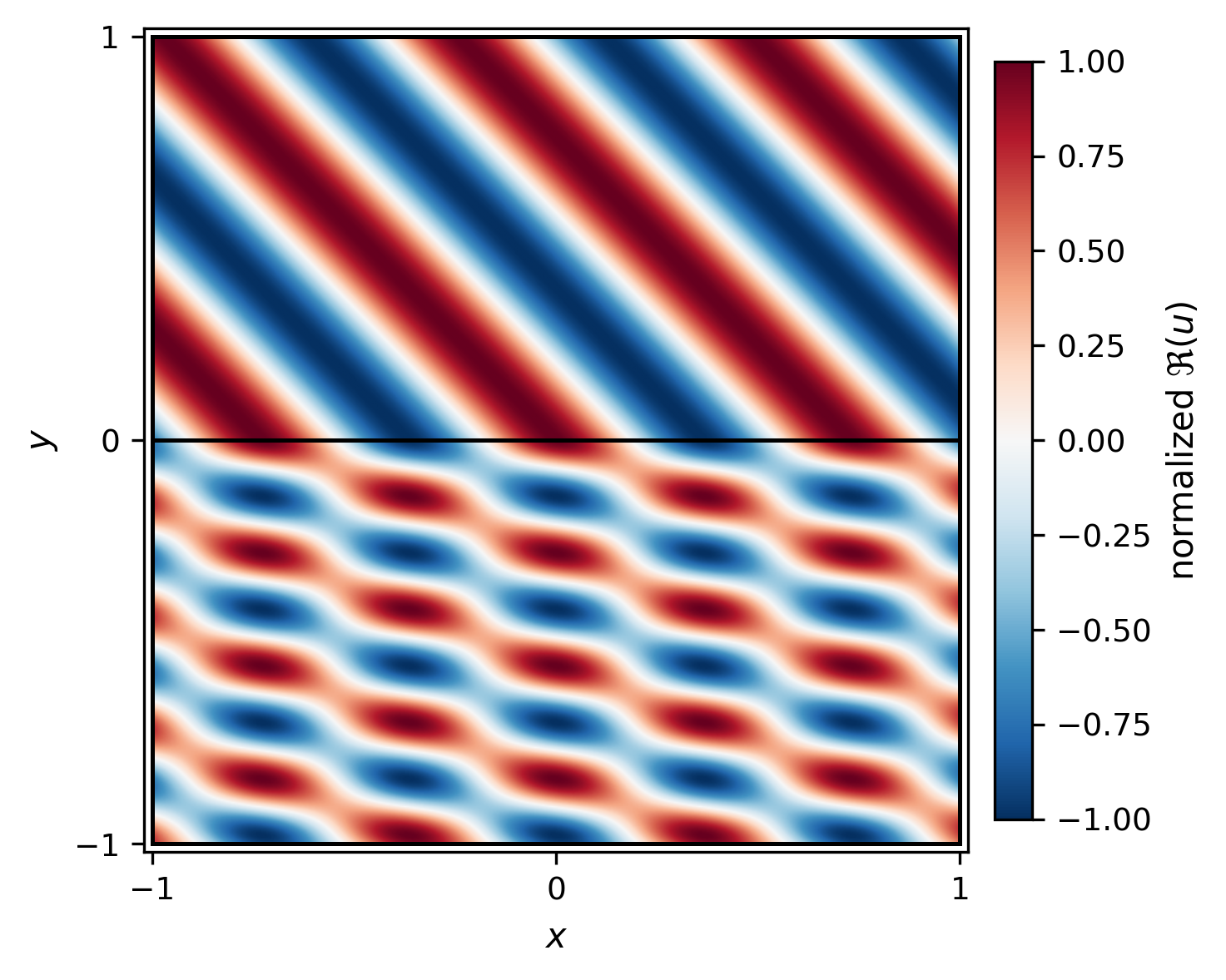}\\[-1mm]
{\footnotesize (b) $\theta_i=69^\circ$}
\end{minipage}\hfill
\begin{minipage}[b]{.29\textwidth}
\centering
\includegraphics[width=\linewidth]{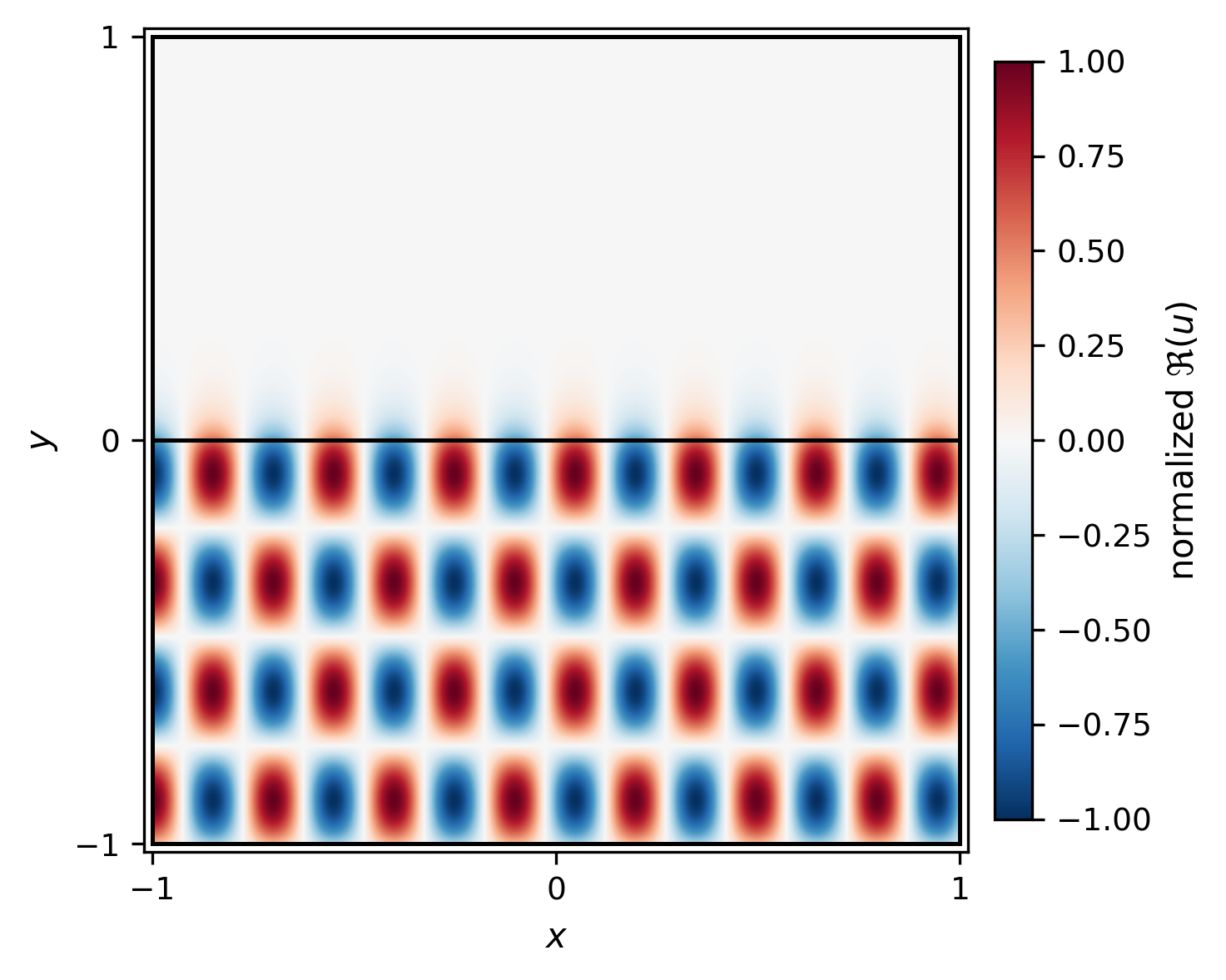}\\[-1mm]
{\footnotesize (c) $\theta_i=29^\circ$}
\end{minipage}
\caption{Representative Helmholtz solutions. (a) A regular Fourier--Bessel solution on the unit disk. (b) Propagating fluid--fluid transmission above the critical angle on $[-1,1]^2$ with interface $y=0$. (c) Complex-angle transmission below the critical angle; the upper-medium solution decays exponentially away from the interface. The geometry in each panel is shown to scale, and colour represents the panelwise-normalized real part.}
\alttext{Alt text: Three normalized real-part heat maps. Left: a fivefold Fourier--Bessel solution on a unit disk. Centre: oscillatory propagating transmission across the horizontal interface $y=0$ at $69$ degrees. Right: at $29$ degrees the lower medium is oscillatory while the transmitted field above $y=0$ decays exponentially.}
\label{fig:solution-gallery}
\end{figure}

\subsection{Verification and local selection}
We compare the closed formulas of \cref{thm:T1,thm:T2,thm:T3} with independent boundary quadrature and direct circle least squares.  \Cref{tab:identitycheck} gives the largest discrepancy.

\begin{table}[htbp]
\centering
\caption{Verification of the three modal identities against independent direct computation.}
\label{tab:identitycheck}
\footnotesize
\begin{tabular}{lll}
\toprule
identity & comparison & discrepancy\\
\midrule
\cref{thm:T1} & analytical $\tau_m^2$ vs.\ direct trace quadrature & $2\times10^{-16}$\\
\cref{thm:T2} & circulant eigenvalues vs.\ direct PW Gram spectrum & $7\times10^{-16}$\\
\cref{thm:T3} & modal tail least squares vs.\ direct circle least squares & ratio $1.0000$\\
\bottomrule
\end{tabular}
\end{table}

We next test the selector on exactly representable sparse, broad, and mixed fields.
On $B(0,0.4)$ at $k=16$, the selector receives only equilibrated modal coefficients and chooses from the same candidate library for three exactly representable fields:
\begin{equation}\label{eq:three-fields}
 u_{\rm ray}=\sum_{j=1}^3c_je^{\ii kd(\theta_j)\cdot x},\;\;
 u_{\rm broad}=\sum_{m=-4}^4b_mJ_m(kr)e^{\ii m\theta},\;\;
 u_{\rm mix}=u_{\rm ray}+0.75\,u_{\rm broad},
\end{equation}
with $(\theta_1,\theta_2,\theta_3)=(18^\circ,71^\circ,143^\circ)$.
The ray count and active harmonics are withheld.  \Cref{tab:selection} recovers the generating family and directions.

\begin{table}[htbp]
\centering
\caption{Blind selection for \eqref{eq:three-fields} at $k=16$.  Active dimension counts retained nonzero local coordinates; the reported error is the relative scaled Cauchy-trace error.}
\label{tab:selection}
\footnotesize
\begin{tabular}{lrrrll}
\toprule
field & $q_{\PW}$ & $q_{\FB}$ & active dim. & rel.\ trace error & recovered directions\\
\midrule
three rays & 3 & 0 & 3 & $6.85\times10^{-16}$ & $18^\circ,71^\circ,143^\circ$\\
broad angular & 0 & 9 & 9 & $0$ & --\\
mixed & 3 & 9 & 12 & $6.71\times10^{-16}$ & $18^\circ,71^\circ,143^\circ$\\
\bottomrule
\end{tabular}
\end{table}

A distributed angular spectrum gives a nontrivial elementwise selection problem on a curved annulus.
The reference field is an outgoing Herglotz--DtN solution.
Let $d(\phi)=(\cos\phi,\sin\phi)$ and prescribe on $r=a$ the trace of the Herglotz wave $v(x)=\int_0^{2\pi}g(\phi)e^{\ii kd(\phi)\cdot x}d\phi$ \citep{ColtonKress2001,ColtonKressBook}.  With $\widehat g_m$ the Fourier coefficients of $g$, Jacobi--Anger gives $v(a,\theta)=2\pi\sum_m\ii^mJ_m(ka)\widehat g_me^{\ii m\theta}$.  Replacing $J_m(kr)$ by $H_m^{(1)}(kr)$ and matching at $r=a$ gives the outgoing annular field
\begin{equation}\label{eq:herglotz-dtn}
 u(r,\theta)=\sum_{m}c_mH_m^{(1)}(kr)e^{\ii m\theta},\qquad
 c_m=\frac{2\pi\ii^mJ_m(ka)\widehat g_m}{H_m^{(1)}(ka)},
\end{equation}
whose Neumann trace on $r=R$ is exactly $\DtN_Nu$ with the multiplier of \eqref{eq:dtn-map}.  The series is truncated at $|m|\le70$, far beyond the active content at $k=16$.  We use
\begin{equation}\label{eq:density}
 g(\phi)=e^{1.55\cos(\phi-0.48)+0.22\cos(2\phi+0.35)}\,
 e^{\ii[0.28\sin(3\phi+0.20)+0.10\cos(5\phi-0.40)]}.
\end{equation}
In contrast to \eqref{eq:transmission}, the field \eqref{eq:herglotz-dtn} has a genuinely distributed angular spectrum and tests exactly that regime (\cref{fig:herglotz}).

\begin{figure}[H]
\centering
\includegraphics[width=.58\textwidth]{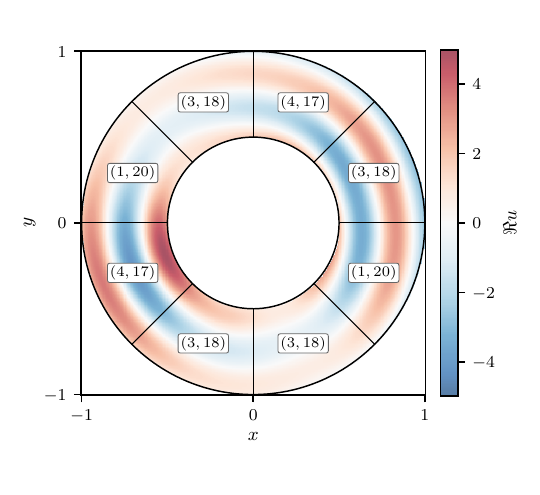}
\caption{Outgoing Herglotz--DtN reference solution \eqref{eq:herglotz-dtn} with density \eqref{eq:density} at $k=16$ on $0.5\le r\le1$. The pair $(q_{\PW},q_{\FB})$ is selected independently on each sector at budget $p=21$; see \cref{tab:blind}.}
\alttext{Alt text: Annular heat map of the outgoing Herglotz--DtN solution with eight sector boundaries. Each sector is labelled by its selected pair of plane-wave and Fourier--Bessel dimensions.}
\label{fig:herglotz}
\end{figure}

For the distributed field \eqref{eq:herglotz-dtn}, use eight curved sectors, $|m|\le70$ in the DtN map, local trace radius $h=0.44$, and budget $p=21$.  Each element independently chooses $0\le q_{\PW}\le4$; the remaining coordinates are centered FB modes.  The reference satisfies the truncated DtN relation to relative residual $2.24\times10^{-15}$.

\begin{table}[htbp]
\centering
\caption{Elementwise selection for the Herglotz--DtN field at $p=21$.  The last column is the relative best local trace error from \eqref{eq:T3}, evaluated before assembly.}
\label{tab:blind}
\footnotesize
\begin{tabular}{rrrrl}
\toprule
element & center angle & $q_{\PW}$ & $q_{\FB}$ & rel.\ local trace error\\
\midrule
0 & $22.5^\circ$ & 3 & 18 & $9.91\times10^{-3}$\\
1 & $67.5^\circ$ & 4 & 17 & $1.29\times10^{-2}$\\
2 & $112.5^\circ$ & 3 & 18 & $4.55\times10^{-2}$\\
3 & $157.5^\circ$ & 1 & 20 & $3.76\times10^{-2}$\\
4 & $202.5^\circ$ & 4 & 17 & $2.89\times10^{-2}$\\
5 & $247.5^\circ$ & 3 & 18 & $8.81\times10^{-2}$\\
6 & $292.5^\circ$ & 3 & 18 & $3.01\times10^{-2}$\\
7 & $337.5^\circ$ & 1 & 20 & $1.58\times10^{-2}$\\
\bottomrule
\end{tabular}
\end{table}

The selected mesh contains $22$ PW and $146$ FB coordinates and retains rank $168/168$, with $\kappa_{\rm tr,final}=1.00000004$ and $\kappa_{\GR}=2.62$.  Its relative global $L^2$ error is $4.33\times10^{-4}$, compared with $6.23\times10^{-4}$ for pure FB and $8.83\times10^{-4}$ for equispaced PW at the same nominal dimension.

\subsection{High-order stability and estimated traces}\label{sec:hankel-highp}
On the annulus $0.5<r<1$ at $k=16$, take the exact outgoing field
\begin{equation}\label{eq:centered-hankel-highp}
 u_{\rm cyl}(x)=H_0^{(1)}(k|x|).
\end{equation}
The mesh consists of eight exact curved sectors and one radial layer.  We impose the exact Dirichlet trace of \eqref{eq:centered-hankel-highp} on $r=0.5$ and the circular DtN condition on $r=1$, truncated at $|m|\le70$.  Each element starts from $p$ equispaced real plane waves.  The local trace disk has center radius $0.75$ and radius $h=0.44$.  Before assembly the normalized plane-wave trace Gram matrix is compressed with the fixed numerical-rank threshold $\tau_{\rm rank}=10^{-12}$, and the retained trace space is orthonormalized.  All quantities in this experiment, including the rank decision, assembly, solve, errors, and condition numbers, are computed in IEEE binary64 arithmetic.

To make the conditioning statements unambiguous, let $G^{\rm tr}_{K,\rm raw}$ be the raw local Cauchy-trace Gram matrix and let $T_K$ be the local transformation returned by the trace-rank/compression step.  We report
\[
 \kappa_{\rm tr,raw}=\max_K\kappa_2(G^{\rm tr}_{K,\rm raw}),
 \qquad
 \kappa_{\rm tr,final}=\max_K\kappa_2(T_K^*G^{\rm tr}_{K,\rm raw}T_K).
\]
With $T=\operatorname{blockdiag}(T_K)$, the assembled matrix in the retained local coordinates is
\[
 K_h=T^*K_{\rm raw}T.
\]
Thus $\kappa_2(K_h)$ is the condition number \emph{after local trace conditioning but before the global graph--Riesz transformation}.  We then form
\[
 G_h=\frac{K_h^*-K_h}{2\ii},\qquad B^*G_hB=I,
\]
and the condition number of the matrix actually solved in graph--Riesz coordinates is
\begin{equation}\label{eq:hankel-highp-kgr}
 \kappa_{\GR}=\kappa_2(B^*K_hB).
\end{equation}
These four condition numbers answer different questions: $\kappa_{\rm tr,raw}$ measures redundancy of the nominal local traces, $\kappa_{\rm tr,final}$ measures conditioning of the retained local trace coordinates, $\kappa_2(K_h)$ measures the assembled DG operator after the local step, and $\kappa_{\GR}$ measures the final globally transformed linear system.

\begin{table}[!htbp]
\centering
\caption{Centered Hankel DtN sweep at $k=16$.  $r_K$ is retained local PW rank and ``retained'' is the global dimension.  The condition columns refer, in order, to the raw local trace Gram, the trace-orthonormalized local Gram, the assembled PWDG matrix, and the final graph--Riesz matrix.}
\label{tab:hankel-highp-global}
\resizebox{\textwidth}{!}{%
\begin{tabular}{rrrrccccc}
\toprule
$p$ & nominal & $r_K$ & retained & rel. $L^2$ error & $\kappa_{\rm tr,raw}$ & $\kappa_{\rm tr,final}$ & $\kappa_2(K_h)$ & $\kappa_{\GR}$\\
\midrule
 9 &  72 &  9 &  72 & $5.0086\times10^{-1}$ & $2.106\times10^{0}$  & $1.0000000$ & $5.646\times10^{0}$ & $1.616$\\
15 & 120 & 15 & 120 & $7.1083\times10^{-2}$ & $1.598\times10^{0}$  & $1.0000000$ & $3.265\times10^{1}$ & $2.106$\\
21 & 168 & 21 & 168 & $7.2243\times10^{-4}$ & $8.627\times10^{1}$  & $1.0000000$ & $4.163\times10^{2}$ & $2.556$\\
27 & 216 & 27 & 216 & $1.1688\times10^{-5}$ & $5.021\times10^{4}$  & $1.0000000$ & $5.728\times10^{3}$ & $3.797$\\
33 & 264 & 33 & 264 & $1.8880\times10^{-7}$ & $1.460\times10^{8}$  & $1.00000004$& $6.793\times10^{4}$ & $4.941$\\
37 & 296 & 37 & 296 & $1.6935\times10^{-8}$ & $6.072\times10^{10}$ & $1.0000158$ & $3.588\times10^{5}$ & $6.173$\\
41 & 328 & 37 & 296 & $1.5608\times10^{-8}$ & $4.108\times10^{13}$ & $1.0000085$ & $3.493\times10^{5}$ & $6.209$\\
49 & 392 & 37 & 296 & $1.5449\times10^{-8}$ & $8.450\times10^{16}$ & $1.0000043$ & $3.429\times10^{5}$ & $6.189$\\
57 & 456 & 37 & 296 & $1.5363\times10^{-8}$ & $1.681\times10^{17}$ & $1.0000039$ & $3.407\times10^{5}$ & $6.146$\\
65 & 520 & 37 & 296 & $1.5517\times10^{-8}$ & $4.932\times10^{17}$ & $1.0000051$ & $3.438\times10^{5}$ & $6.161$\\
73 & 584 & 37 & 296 & $1.5368\times10^{-8}$ & $2.806\times10^{17}$ & $1.0000022$ & $3.395\times10^{5}$ & $6.148$\\
81 & 648 & 37 & 296 & $1.5372\times10^{-8}$ & $3.189\times10^{17}$ & $1.0000045$ & $3.411\times10^{5}$ & $6.151$\\
89 & 712 & 37 & 296 & $1.5326\times10^{-8}$ & $3.038\times10^{17}$ & $1.0000020$ & $3.388\times10^{5}$ & $6.131$\\
97 & 776 & 37 & 296 & $1.5391\times10^{-8}$ & $4.151\times10^{17}$ & $1.0000031$ & $3.374\times10^{5}$ & $6.125$\\
\bottomrule
\end{tabular}}
\end{table}

\begin{figure}[H]
\centering
\begin{minipage}[t]{.30\textwidth}\centering
\includegraphics[width=\linewidth]{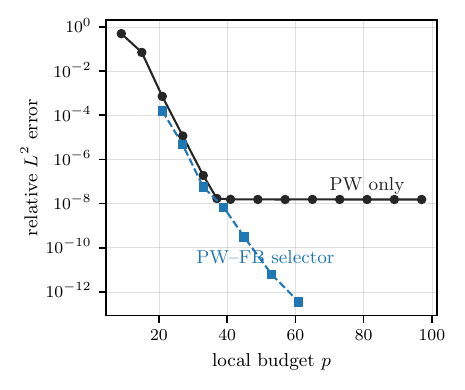}
\end{minipage}\hfill
\begin{minipage}[t]{.30\textwidth}\centering
\includegraphics[width=\linewidth]{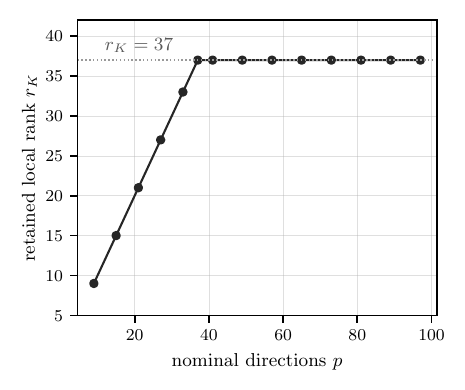}
\end{minipage}\hfill
\begin{minipage}[t]{.30\textwidth}\centering
\includegraphics[width=\linewidth]{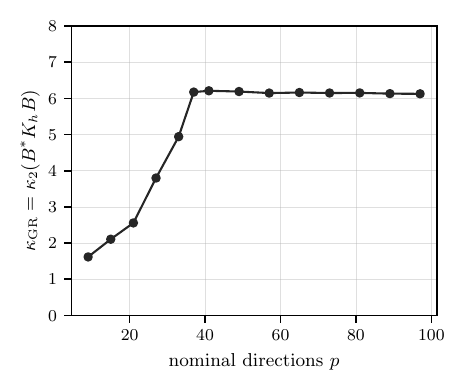}
\end{minipage}
\caption{Centered-Hankel DtN experiment. Left: global $L^2$ error. Centre: retained plane-wave rank. Right: $\kappa_{\GR}$. The plane-wave-only error stalls when the retained local rank reaches $37$, whereas PW--FB selection passes this floor.}
\alttext{Alt text: Three plots versus local budget $p$. Left: global relative $L^2$ error for plane-wave only and PW--FB selection. Centre: retained plane-wave rank rises and then plateaus at $37$. Right: the graph--Riesz condition number remains moderate as $p$ increases.}
\label{fig:hankel-highp-global}
\end{figure}

The table exposes the practical saturation mechanism directly.  Up to $p=37$ every nominal direction is retained and the global error falls from $5.01\times10^{-1}$ to $1.69\times10^{-8}$.  At $p=41$ the rank test first rejects unresolved directions: the nominal dimension grows to $328$, but the retained dimension stays at $296$.  From that point through $p=97$, every element retains exactly $37$ directions and the error remains near $1.5\times10^{-8}$.  The high nominal dimensions therefore provide increasingly redundant raw descriptions of essentially the same numerically resolvable local trace space.

The PW-only error floor coincides with local rank saturation at $r_K=37$: $\kappa_{\rm tr,final}\approx1$ and $\kappa_{\GR}\approx6.1$ even after $\kappa_{\rm tr,raw}$ enters the binary64 singularity range.  The obstruction is therefore local trace resolution, not conditioning of the final solve.

To determine whether Fourier--Bessel coordinates pass the plane-wave ceiling, we repeat the same DtN problem with $q\in\{0,2,4,6\}$ and an FB remainder at fixed budget $p$.  Every candidate is polished in the weighted tail norm and must retain full hybrid trace rank at $\varepsilon_{\rm solve}=10^{-12}$.

\begin{table}[!htbp]
\centering
\caption{PW--FB continuation for the centered Hankel DtN problem.  The FB block is $|m|\le M$; retained dimension is measured after local trace orthonormalization.}
\label{tab:hankel-hybrid-continuation}
\footnotesize
\begin{tabular}{rrrrrcc}
\toprule
$p$ & $q_{\PW}$ & $M$ & retained & rel. $L^2$ error & $\kappa_2(K_h)$ & $\kappa_{\GR}$\\
\midrule
21 & 2 &  9 & 21 & $1.5982\times10^{-4}$ & $7.232\times10^{2}$ & 3.023\\
27 & 2 & 12 & 27 & $4.6275\times10^{-6}$ & $1.024\times10^{4}$ & 3.329\\
33 & 2 & 15 & 33 & $5.6434\times10^{-8}$ & $1.148\times10^{5}$ & 5.732\\
39 & 0 & 19 & 39 & $6.8243\times10^{-9}$ & $7.132\times10^{5}$ & 6.368\\
45 & 0 & 22 & 45 & $3.0505\times10^{-10}$ & $8.613\times10^{6}$ & 7.909\\
53 & 0 & 26 & 53 & $6.3124\times10^{-12}$ & $2.063\times10^{8}$ & 10.685\\
61 & 0 & 30 & 61 & $3.5005\times10^{-13}$ & $5.416\times10^{9}$ & 13.152\\
\bottomrule
\end{tabular}
\end{table}

At $p=21,27,33$ the selector retains two PW directions and an FB block.  From $p=39$ onward the lowest-tail mixed candidates fail the full hybrid rank test, so the selected family becomes pure FB.  The error then passes the PW-only rank-$37$ floor and reaches $3.50\times10^{-13}$ at $p=61$, while the final matrix has $\kappa_{\GR}=13.15$.

The rank-threshold sweep in \cref{tab:hankel-rank-sensitivity} shows a stable plateau: $10^{-11}$ and $10^{-12}$ retain the same $37$ directions and give identical errors, whereas $10^{-13}$ admits additional directions for which $G_h$ is no longer positive definite in binary64.

\begin{table}[!htbp]
\centering
\caption{Rank-threshold sensitivity at $p=97$.  ``Fail'' means that the assembled $G_h$ is not positive definite in binary64.}
\label{tab:hankel-rank-sensitivity}
\footnotesize
\begin{tabular}{ccccc}
\toprule
$\tau_{\rm rank}$ & retained $r_K$ & rel. $L^2$ error & $\kappa_2(K_h)$ & $\kappa_{\GR}$\\
\midrule
$10^{-10}$ & 35 & $4.6080\times10^{-8}$ & $1.582\times10^5$ & 5.864\\
$10^{-11}$ & 37 & $1.5391\times10^{-8}$ & $3.374\times10^5$ & 6.125\\
$10^{-12}$ & 37 & $1.5391\times10^{-8}$ & $3.374\times10^5$ & 6.125\\
$10^{-13}$ & $>37$ & -- & -- & Fail\\
\bottomrule
\end{tabular}
\end{table}

An auxiliary arbitrary-precision local calculation confirms that the nominal PW span continues beyond this binary64 rank ceiling.  \Cref{tab:hankel-highp-global,tab:hankel-rank-sensitivity} concern the realizable binary64 algorithm, including rank selection, assembly, and the final graph--Riesz solve.

The disk radius and rank tolerance are coupled through $\kappa h$.  At $p=97$ the geometric containing radius of one annular sector is $0.4204$; \cref{tab:radius-sensitivity} varies $h$ just above this value and farther outward.
\begin{table}[htbp]
\centering
\caption{Trace-radius/rank-threshold sensitivity for the $p=97$ centered-Hankel test.  ``Fail'' means that the retained floating-point coordinates do not give a positive graph matrix.}
\label{tab:radius-sensitivity}
\footnotesize
\begin{tabular}{ccccc}
\toprule
$h$ & $\tau_{\rm rank}$ & retained $r_K$ & rel. $L^2$ error & $\kappa_{\GR}$\\
\midrule
0.425 & $10^{-12}$ & 37 & $1.538\times10^{-8}$ & 6.128\\
0.440 & $10^{-12}$ & 37 & $1.539\times10^{-8}$ & 6.125\\
0.460 & $10^{-12}$ & -- & Fail & --\\
0.460 & $10^{-11}$ & 37 & $1.540\times10^{-8}$ & 6.127\\
0.500 & $10^{-11}$ & -- & Fail & --\\
0.500 & $10^{-10}$ & 37 & $1.540\times10^{-8}$ & 6.126\\
\bottomrule
\end{tabular}
\end{table}
Thus the physical approximation is essentially unchanged when the same $37$ trace directions are retained.  Increasing $h$ moves more modal content above a fixed relative threshold, so the threshold must be tightened to the accuracy that the global floating-point solve can actually support.  This is the practical reason for choosing a near-minimal admissible disk and for tying $\tau_{\rm rank}$ to $\varepsilon_{\rm solve}$.

We next replace exact selector data by modal data estimated from a coarse numerical field rather than from the exact trace.  Starting from nine equispaced PW plus $J_0$ (relative $L^2$ error $3.7\times10^{-1}$), we estimate modal data, select at budget $p=21$, solve, and repeat once.  Element-boundary Cauchy least squares uses a relative SVD cutoff $10^{-10}$.  As a comparison, an inscribed-circle estimator at $h_{\rm in}=0.24$ inverts factors $J_m(\kappa h_{\rm in})$ and therefore amplifies high-mode perturbations by $|J_m(\kappa h_{\rm in})|^{-1}$.

\begin{table}[!htbp]
\centering
\caption{Selection from estimated trace data at $p=21$.  Data error is measured in equilibrated modal coordinates; the control uses the exact trace.}
\label{tab:estimated-trace}
\footnotesize
\begin{tabular}{llccll}
\toprule
stage & estimator & data error & $(q_{\PW},M)$ & rel.\ $L^2$ error & $\kappa_{\GR}$\\
\midrule
coarse solve & -- & -- & -- & $3.72\times10^{-1}$ & --\\
cycle 1 & element boundary & $7.9\times10^{-1}$ & $(2,9)$ & $2.24\times10^{-4}$ & 2.88\\
cycle 2 & element boundary & $1.9\times10^{-2}$ & $(2,9)$ & $2.24\times10^{-4}$ & 2.88\\
cycle 1 & inscribed circle & $7.9\times10^{-1}$ & $(6,7)$ & $7.80\times10^{-4}$ & 2.63\\
cycle 2 & inscribed circle & $4.3\times10^{-2}$ & $(6,7)$ & $8.50\times10^{-4}$ & 2.55\\
control & exact trace & $0$ & $(2,9)$ & $1.60\times10^{-4}$ & 3.02\\
\bottomrule
\end{tabular}
\end{table}

Boundary-estimated data select the same $(2,9)$ space as the exact-trace control after one cycle and give error $2.24\times10^{-4}$ versus $1.60\times10^{-4}$ for the control.  The inscribed-circle data select $(6,7)$ and give $7.80\times10^{-4}$, consistent with the Bessel-factor amplification above.

To test the fallback without a containing disk, consider the L-shaped domain $(-1,1)^2\setminus((0,1)\times(-1,0))$ and the corner-singular Helmholtz solution
\begin{equation}\label{eq:lshape-singular}
 u(r,\theta)=J_{2/3}(\kappa r)\sin(2\theta/3),\qquad 0\le\theta\le3\pi/2,\qquad \kappa=12.
\end{equation}
A representative true solution is shown in \cref{fig:lshape-solution}.  The panel is a true 2D projection on the physical L-shaped domain, displayed to scale with the domain boundary and a visible colormap for the panelwise-normalized real part.

\begin{figure}[H]
\centering
\includegraphics[width=.58\textwidth]{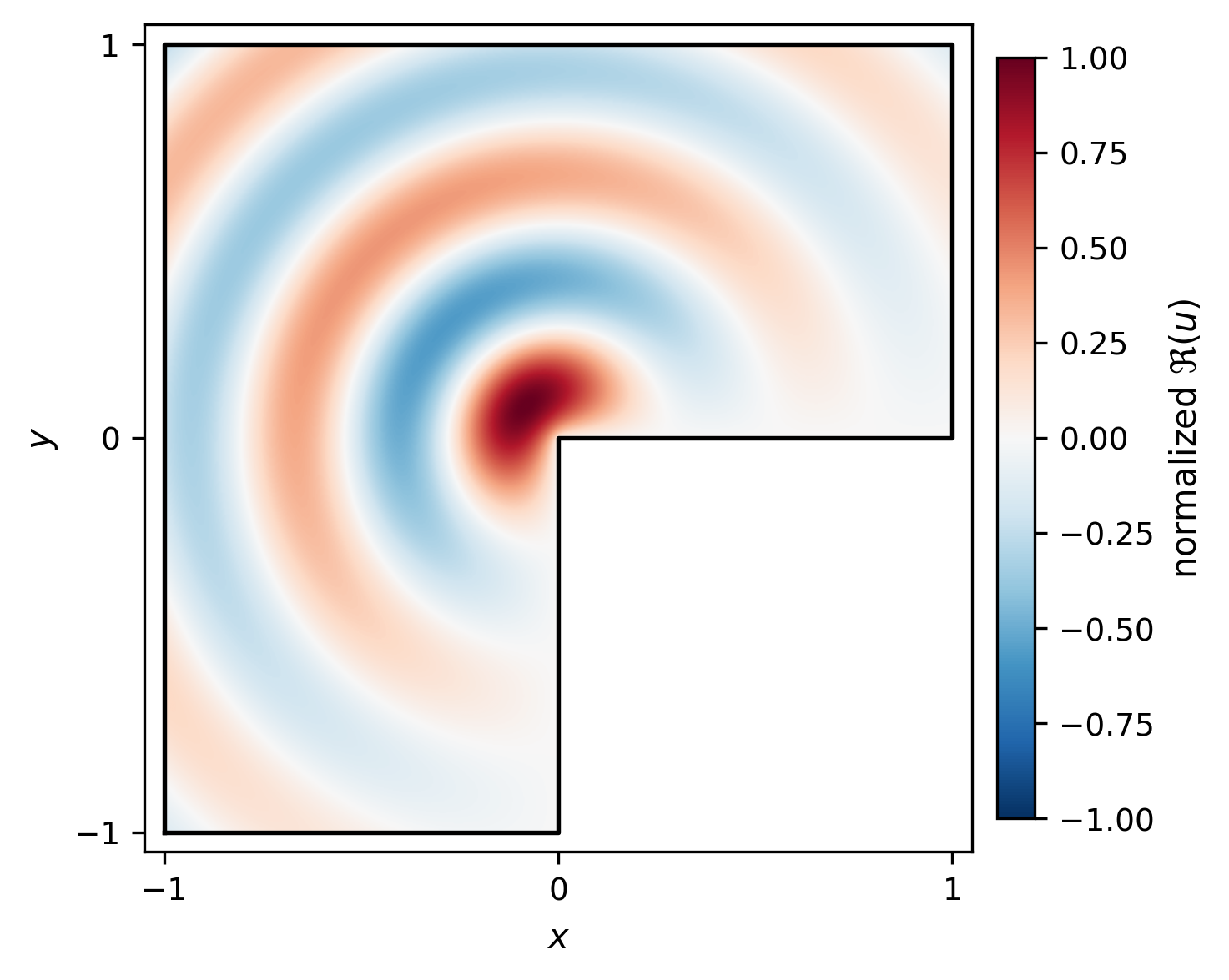}
\caption{True solution of the L-shaped test problem \eqref{eq:lshape-singular} on $(-1,1)^2\setminus((0,1)\times(-1,0))$. The geometry is shown to scale, and colour represents the panelwise-normalized real part.}
\alttext{Alt text: Heat map of the normalized real part of the corner-singular Helmholtz solution on an L-shaped domain obtained by removing the lower-right quadrant from the square $[-1,1]^2$. Concentric wave fronts emanate from the re-entrant corner at the origin.}
\label{fig:lshape-solution}
\end{figure}

This solution is regular away from the reentrant vertex but is not analytic there.  On every triangle the selector therefore fits the exact Cauchy data on $\partial K$ directly.  At fixed local budget $p=7$ it compares $(q_{\PW},q_{\FB})=(0,7)$ and $(2,5)$; the two PW angles are refined by local variable projection.  No global residual search is used.
\begin{table}[htbp]
\centering
\caption{Boundary-trace selection on the L-shaped corner problem \eqref{eq:lshape-singular}.  The last three columns are global relative $L^2$ errors at the same local budget $p=7$.}
\label{tab:lshape-boundary}
\footnotesize
\begin{tabular}{rrrrrrr}
\toprule
$n$ & triangles & FB cells & hybrid cells & selector & pure FB & equispaced PW\\
\midrule
2 & 24 & 4 & 20 & $7.35\times10^{-1}$ & $1.23$ & $9.52\times10^{-1}$\\
3 & 54 & 2 & 52 & $2.23\times10^{-1}$ & $6.01\times10^{-1}$ & $6.43\times10^{-1}$\\
\bottomrule
\end{tabular}
\end{table}
The corresponding graph--Riesz condition numbers are $4.61$ and $7.65$.  This is a coarse robustness test, not a singular-corner convergence theorem: its purpose is to show that the selector remains well defined when \eqref{eq:radius-window} fails and the comparison is made on the physical element boundary.

\subsection{Sparse recovery, complex transmission, and search cost}
For the three-ray field, ESPRIT recovers the generating directions to about $3\times10^{-11}$ radians.  By \cref{cor:sparse}, any PW space containing those directions is exact; \cref{tab:sparse} confirms roundoff-level global error, while equispaced PW and pure FB at the same nominal dimensions remain unresolved.

\begin{table}[!htbp]
\centering
\caption{Global relative $L^2$ error for the three-ray field at $\kappa=8$.  Pure FB uses $(0,p)$; equispaced and direction-identified PW both use $(p,0)$, so the last two columns isolate the choice of directions.}
\label{tab:sparse}
\footnotesize
\begin{tabular}{rccc}
\toprule
$p$ & FB $(0,p)$ & equispaced PW $(p,0)$ & identified PW $(p,0)$\\
\midrule
5  & $3.41\times10^{-1}$ & $4.73\times10^{-1}$ & $5.30\times10^{-13}$\\
9  & $2.87\times10^{-2}$ & $4.56\times10^{-2}$ & $8.11\times10^{-15}$\\
13 & $9.31\times10^{-4}$ & $1.60\times10^{-3}$ & $8.94\times10^{-14}$\\
\bottomrule
\end{tabular}
\end{table}
Once the three recovered rays are contained in the local space, the exact solution is representable.  The variations between $10^{-15}$ and $10^{-13}$ in the selector column are therefore floating-point assembly/solve effects and are not expected to decrease monotonically with the nominal budget.

The interface problem now tests the same selector with complex angles.  On the eight-triangle mesh take $n_1=2$, $n_2=1$, $\omega=12$, $\alpha=\beta=1/2$, trace radius $h=0.44$, and two PWs below the interface and one above.  Exact Cauchy traces are supplied to isolate the representation problem.  The physical angles and the critical angle are not supplied to ESPRIT.  For each incidence angle the lower trace recovers the incident and reflected real directions; the upper trace recovers one node $\zeta_t$ and hence $z_t=\ii\operatorname{Log}\zeta_t$.

\begin{figure}[H]
\centering
\begin{minipage}[t]{.45\textwidth}
\centering
\includegraphics[width=\linewidth]{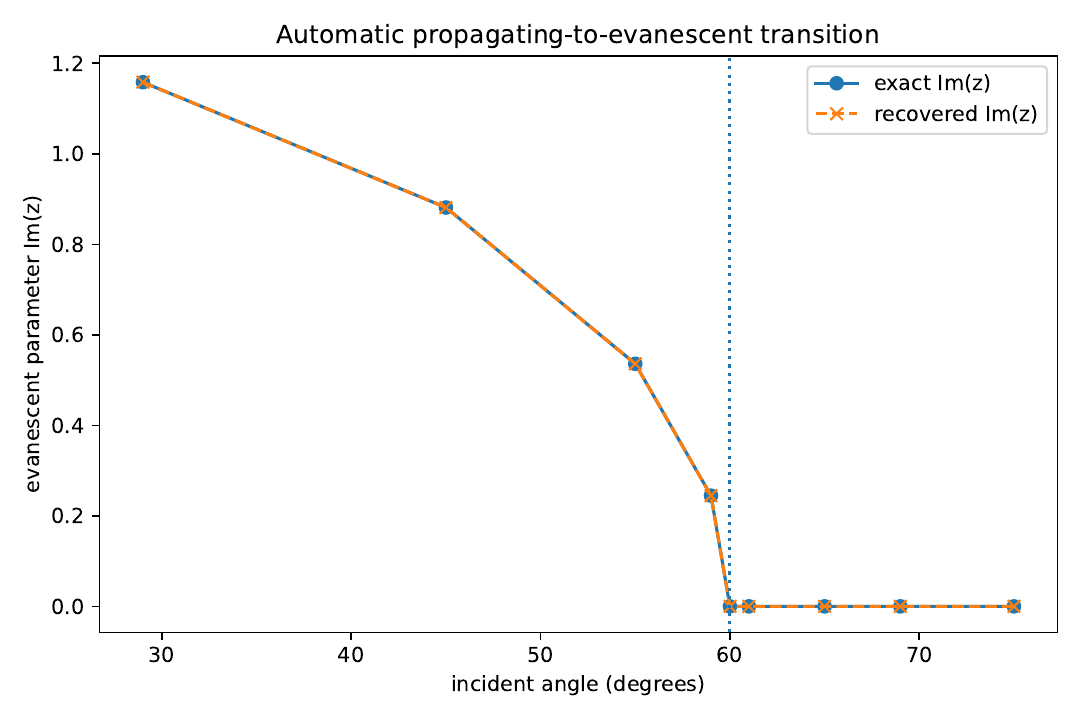}\\[-1mm]
{\footnotesize (a) recovered evanescent parameter}
\end{minipage}\hfill
\begin{minipage}[t]{.45\textwidth}
\centering
\includegraphics[width=\linewidth]{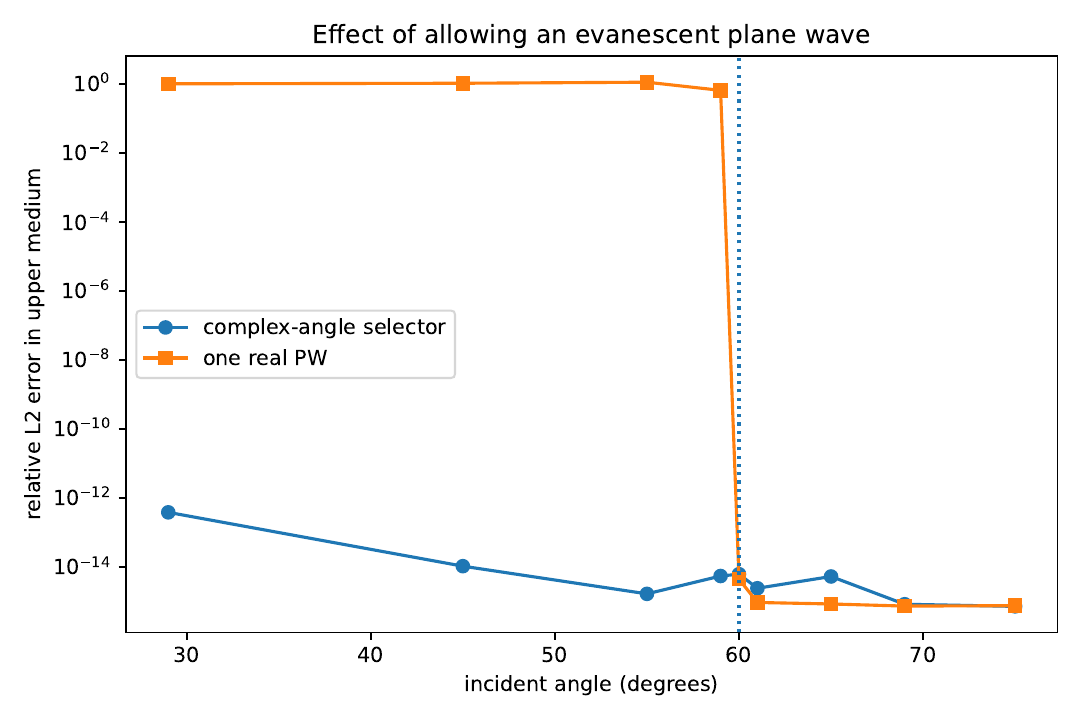}\\[-1mm]
{\footnotesize (b) complex versus real transmitted PW}
\end{minipage}
\caption{Automatic transition in the fluid--fluid problem \eqref{eq:transmission}. The dotted line is $\theta_c=60^\circ$. Left: exact and recovered $\Im z_t$. Right: relative $L^2$ error in the upper medium using the recovered complex angle and using the same one-PW model constrained to a real angle. No critical-angle test is used by the selector.}
\alttext{Alt text: Two plots versus incident angle. Left: exact and recovered imaginary transmitted angle agree through the $60$-degree critical transition. Right: the complex-angle representation keeps the upper-medium error near roundoff, while a real-only transmitted PW has order-one error below criticality.}
\label{fig:evanescent-transition}
\end{figure}

The transition is resolved to essentially machine precision.  Representative values are given in \cref{tab:transmission}.  In particular, at $29^\circ$,
\[
 z_t=1.1582805855\,\ii,
\]
and ESPRIT returns the same imaginary part to $6.4\times10^{-14}$.  The full PWDG error is $5.18\times10^{-14}$ with $\kappa_{\GR}=2.25$.  Constraining the transmitted PW to be real gives upper-medium relative error $1.03$; the error is representational, not a failure of the graph-normalized global solve.

\begin{table}[htbp]
\centering
\caption{Automatic propagating-to-evanescent transmission recovery.  $z_t$ is the exact complex angle from \eqref{eq:transmitted-complex-angle}; $\widehat z_t$ is recovered from the local modal trace.  The real-PW column is the upper-medium relative $L^2$ error when the transmitted PW is constrained to $\Im z=0$.}
\label{tab:transmission}
\footnotesize
\begin{tabular}{rllllrr}
\toprule
$\theta_i$ & regime & $z_t$ & $\widehat z_t$ & global rel. $L^2$ & real-PW upper error & $\kappa_{\GR}$\\
\midrule
$69^\circ$ & prop. & $44.214356^\circ$ & $44.214356^\circ$ & $9.77\times10^{-16}$ & $7.29\times10^{-16}$ & 1.84\\
$60^\circ$ & grazing & $0$ & $3.6\times10^{-17}\ii$ & $8.49\times10^{-15}$ & $4.43\times10^{-15}$ & 1.77\\
$59^\circ$ & evan. & $0.244649\ii$ & $0.244649\ii$ & $4.97\times10^{-15}$ & $6.67\times10^{-1}$ & 1.77\\
$45^\circ$ & evan. & $0.881374\ii$ & $0.881374\ii$ & $3.20\times10^{-14}$ & $1.07$ & 2.45\\
$29^\circ$ & evan. & $1.158281\ii$ & $1.158281\ii$ & $5.18\times10^{-14}$ & $1.03$ & 2.25\\
\bottomrule
\end{tabular}
\end{table}

Finally, we compare local trace selection with the globally coupled residual search on the same approximation family.
On the Herglotz--DtN field, both methods optimize the same three shared real PW directions from the same start.  The global objective \eqref{eq:constrained-opt} assembles and solves PWDG at every trial; the local objective \eqref{eq:modal-varpro} performs only elementwise least squares and one final PWDG solve.  FB functions are excluded so that the comparison isolates the search objective.  The timing comparison likewise isolates this repeated-search cost: any trace sampling or modal projection needed to construct the local data is a one-time preprocessing cost per update and is not included in either nonlinear trial time.  Thus \cref{tab:timing,tab:timing-levels} should be read as a comparison of search stages, not as a universal end-to-end speedup when trace acquisition itself is expensive.

\begin{table}[htbp]
\centering
\caption{Cost of global-residual and local-trace selection on the same three-direction Herglotz--DtN problem at $k=16$.  Total time includes the final DG solve.}
\label{tab:timing}
\footnotesize
\begin{tabular}{lrrrrc}
\toprule
method & evals & selection (s) & final DG (s) & total (s) & rel.\ $L^2$\\
\midrule
global residual $\Jfun$ & 271 & 10.428 & 0.145 & 10.573 & $7.346\times10^{-1}$\\
local Cauchy trace & 240 & 0.346 & 0.037 & 0.383 & $7.383\times10^{-1}$\\
\bottomrule
\end{tabular}
\end{table}

The two errors differ by $0.5\%$, while total search time differs by a factor $27.6$.  The gain is per trial: $38$\,ms for the globally coupled objective versus $1.4$\,ms for the local trace objective.

Repeating the comparison under mesh refinement separates the element-local and globally coupled costs.
We repeat the same three-direction Herglotz--DtN search on $(n_\theta,n_r)=(8,1),(16,2),(32,4)$, using identical starts, tolerances, and nonlinear iteration.  Local trials contain only independent element problems; global trials assemble and solve the PWDG system.

\begin{table}[htbp]
\centering
\caption{Search-cost scaling at $k=16$ on one 2.1\,GHz Xeon core.  Global trials assemble and solve PWDG; local trials solve only elementwise trace problems.}
\label{tab:timing-levels}
\footnotesize
\begin{tabular}{rrrrrrrrr}
\toprule
elements & dofs & \multicolumn{2}{c}{per trial (ms)} & \multicolumn{2}{c}{total (s)} & ratio & \multicolumn{2}{c}{rel.\ $L^2$}\\
 & & global & local & global & local & & global & local\\
\midrule
  8 &  24 &   6.3 & 0.31 &  0.946 & 0.075 & 12.6 & $0.780$ & $0.743$\\
 32 &  96 &  34.2 & 1.12 &  6.810 & 0.254 & 26.8 & $0.564$ & $0.575$\\
128 & 384 & 307.8 & 3.93 & 58.761 & 1.240 & 47.4 & $0.426$ & $0.460$\\
\bottomrule
\end{tabular}
\end{table}

The local per-trial cost scales approximately with element count, while the total global/local ratio grows from $12.6$ to $47.4$; the final errors remain within $8\%$ on all three levels.

\subsection{Broad and mixed angular content}\label{sec:broad-content}
For a broad $96$-direction random-phase Herglotz field, sparse direction identification provides no structural advantage.  At moderate dimension, pure FB and equispaced PW have comparable angular reach; at high dimension their difference is stability, because dense PW traces become correlated while equilibrated FB modes remain orthogonal.  \Cref{tab:p17-fb-trace} records the FB truncation scale at $p=17$.

\begin{table}[htbp]
\centering
\caption{Relative Cauchy-trace error of the pure FB truncation at $p=17$.}
\label{tab:p17-fb-trace}
\footnotesize
\begin{tabular}{lc}
\toprule
spectrum & relative trace error at $p=17$\\
\midrule
sparse directional & $8.4\times10^{-5}$\\
broad angular & $9.3\times10^{-5}$\\
mixed & $1.0\times10^{-4}$\\
\bottomrule
\end{tabular}
\end{table}

At high order, trace equilibration becomes decisive for Fourier--Bessel coordinates.
Without FB trace equilibration, the retained global rank stalls at $184$ and the error remains near $2.5\times10^{-8}$.  Equilibration retains the full modal dimension and restores high-order convergence (\cref{tab:equil}); this is the scaling effect predicted by \cref{thm:T1,rem:raw-coeff}.

\begin{table}[htbp]
\centering
\caption{High-order pure FB spaces $(0,p)$ at $\kappa=8$ with and without trace equilibration.  Parentheses give retained over nominal global graph rank.}
\label{tab:equil}
\footnotesize
\begin{tabular}{rcc}
\toprule
$p$ & FB without equilibration & FB with trace equilibration\\
\midrule
25 & $2.45\times10^{-8}$ $(184/200)$ & $2.69\times10^{-9}$ $(200/200)$\\
33 & $2.45\times10^{-8}$ $(184/264)$ & $1.16\times10^{-13}$ $(264/264)$\\
\bottomrule
\end{tabular}
\end{table}

For mixed fields, \eqref{eq:crossover-poly} gives the field-dependent exchange points between prescribed candidate spaces.  The selector recovers $(3,9)$ in \cref{tab:selection}, while \cref{tab:blind} shows spatially varying mixed choices that improve on both pure families at equal nominal dimension.

\subsection{ESPRIT capacity and conditioning}\label{sec:esprit-capacity}
This section isolates algebraic direction identification from the DG approximation.  Exact modal data are supplied, so the only limits are sample count, Hankel conditioning, and floating-point arithmetic.

We first fix the modal window and increase the recovered rank.  The residual continuation of \citet{Kapita2026Directions} recovered its nested ray family automatically through $q=19$; its $q=20$ birth entered a false basin, while a nearby-birth control recovered the same $20$-ray space to roundoff.  With $N=81$ exact modal samples, direct ESPRIT instead recovers every admissible rank $q\le40$.  At the endpoint $q=40$, the shifted matrix $U_0$ is square ($40\times40$) and nonsingular; ESPRIT does not require a nullspace of the Hankel matrix.  Requesting $q=41$ violates the shift-dimension condition $q\le(N-1)/2$: the shifted signal matrix has only $40$ rows and rank $40$, and the resulting maximum angular error is $25.30^\circ$.

For $N=401$, the balanced Hankel matrix is $201\times201$ and \cref{thm:esprit-exact} permits $q\le200$.  We retain the first twenty directions and coefficient phases of \citet{Kapita2026Directions}; subsequent directions are added deterministically by farthest-point insertion on a $0.05^\circ$ angular grid, with phases $\phi_j=\operatorname{mod}(0.731j+0.173j^2,2\pi)-\pi$ and coefficients $e^{\ii\phi_j}$.

The error remains near angular roundoff through more than $160$ rays and is $8.5\times10^{-9}$ degrees at $q=184$, where $\kappa_H:=\sigma_1(H)/\sigma_q(H)=6.27\times10^4$.  At $q=194$, $\kappa_H=4.34\times10^8$ and the error is $1.08\times10^{-4}$ degrees; at $q=198$, $\kappa_H=2.52\times10^{12}$ and the error is $1.65^\circ$.  Thus the high-rank failure is a conditioning limit of the chosen modal window, not a representation limit.

\begin{figure}[H]
\centering
\includegraphics[width=.68\textwidth]{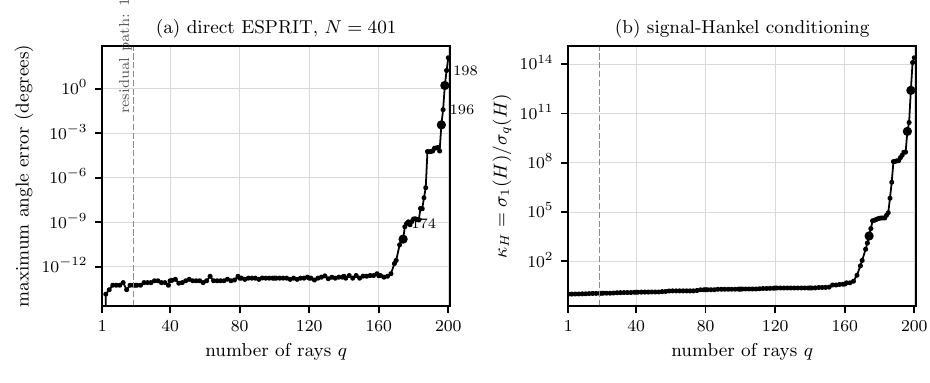}
\caption{Direct ESPRIT with $N=401$ target-frequency modal samples. Recovery remains close to roundoff through more than $160$ rays; failure near $q=198$ coincides with growth of $\kappa_H=\sigma_1(H)/\sigma_q(H)$.}
\alttext{Alt text: Two ESPRIT capacity plots versus number of rays $q$ for $N=401$ modal samples. Left: maximum angle error stays near roundoff until roughly $160$ rays and then rises rapidly near $q=198$. Right: the Hankel condition number shows corresponding rapid growth.}
\label{fig:esprit-capacity}
\end{figure}

The rank limit moves with the modal-window length.  For example, $q=900$ at $N=2001$ has $\kappa_H=1.00\times10^{12}$ and $1.27\times10^{-1}$ degree error, whereas the same rank at $N=2201$ has $\kappa_H=2.16\times10^2$ and $9.46\times10^{-12}$ degree error (\cref{tab:esprit-conditioning}).  This is consistent with Vandermonde conditioning results for separated nodes \citep{Moitra2015,LiLiao2020}.
The same mechanism governs practical weak or nonideal ray data.  With $N=81$ and three target rays, \cref{tab:esprit-practical} isolates angular separation, amplitude imbalance, and a deterministic $96$-direction background.
\begin{table}[htbp]
\centering
\caption{Finite-ray ESPRIT sensitivity with $N=81$ modal samples.  The background amplitude $\rho$ multiplies a unit-energy $96$-direction random-phase field.}
\label{tab:esprit-practical}
\footnotesize
\begin{tabular}{lrrr}
\toprule
perturbation & parameter & max. angle error (deg.) & $\kappa_H$\\
\midrule
ray separation & $1^\circ$ & $2.84\times10^{-14}$ & $9.28\times10^1$\\
ray separation & $0.1^\circ$ & $5.15\times10^{-13}$ & $9.26\times10^3$\\
weak-ray amplitude & $10^{-4}$ & $2.27\times10^{-13}$ & $1.00\times10^4$\\
weak-ray amplitude & $10^{-8}$ & $3.67\times10^{-9}$ & $1.00\times10^8$\\
diffuse background & $\rho=10^{-2}$ & $6.21\times10^{-3}$ & $2.00$\\
diffuse background & $\rho=3\times10^{-1}$ & $2.03\times10^{-1}$ & $1.94$\\
\bottomrule
\end{tabular}
\end{table}
The first two blocks show the expected loss of identifiability through $\sigma_q(H)$.  The diffuse case is different: the data are no longer exactly rank three, so $\kappa_H$ measures conditioning of the retained three-dimensional subspace but not closeness of the data to a three-exponential model.  The nonzero angle shift at $\rho=10^{-2}$ despite $\kappa_H\approx2$ is therefore model mismatch rather than ill-conditioning, consistent with \cref{thm:esprit-stability,rem:bridge-general}.  In the selector this case is not accepted on $\kappa_H$ alone; the variable-projection residual and the final trace error remain part of the comparison.  Practically, a proposed ray is retained only while its signal singular value is above the trace-data floor and the associated reduction of the weighted modal residual is commensurate with the requested accuracy.  MUSIC or sparse dictionary methods can be substituted for the initializer, but they do not change the trace metric, hybrid approximation identity, or rank test developed here.

For an $L\times K$ Hankel matrix, a dense SVD costs $O(LK\min\{L,K\})$; a truncated rank-$q$ factorization reduces this to $O(LKq)$ when $q\ll\min\{L,K\}$.  The work is element local and parallel.  The large windows $N=401$--$2201$ below are conditioning stress tests, not default selector sizes; production windows are chosen only large enough to satisfy $L\ge q+1$, $K\ge q$, and to keep the retained $\tau_m$ above the trace-data floor.  Since a resolved angular bandwidth typically grows with $\kappa h$, very high-frequency patches can make a large dense ESPRIT solve unattractive; this is precisely the regime in which a truncated factorization is preferable for sparse content, while broad angular spectra favor FB coordinates.

For the ESPRIT shift fit define
\begin{equation}\label{eq:kappa-sh}
 \kappa_{\rm sh}:=\kappa_2(U_0)=\frac{\sigma_1(U_0)}{\sigma_q(U_0)}.
\end{equation}
The table shows that $\kappa_{\rm sh}$ remains modest until the Hankel signal space is already ill-conditioned.  This matches \cref{thm:esprit-stability}: the leading amplification is $\sigma_q(H)^{-1}$, followed by the pseudoinverse factor $\sigma_q(U_0)^{-1}$.

\begin{table}[!htbp]
\centering
\caption{Representative ESPRIT conditioning data.  Increasing the modal window can restore a fixed rank by reducing $\kappa_H$.}
\label{tab:esprit-conditioning}
\footnotesize
\begin{tabular}{rrrrr}
\toprule
$N$ & $q$ & max. angle error (deg.) & $\kappa_H$ & $\kappa_{\rm sh}$\\
\midrule
1001 & 400 & $2.27\times10^{-13}$ & $3.66$ & $1.50$\\
1001 & 470 & $8.16\times10^{-3}$ & $8.04\times10^{11}$ & $36.5$\\
1401 & 580 & $5.95\times10^{-8}$ & $9.17\times10^{5}$ & $4.33$\\
2001 & 900 & $1.27\times10^{-1}$ & $1.00\times10^{12}$ & $9.64$\\
2201 & 900 & $9.46\times10^{-12}$ & $2.16\times10^{2}$ & $1.75$\\
2201 & 940 & $1.70\times10^{-1}$ & $2.93\times10^{12}$ & $4.31$\\
\bottomrule
\end{tabular}
\end{table}
\FloatBarrier

\section{Conclusions}\label{sec:conclusion}
The central point of this work is that local representation quality and global algebraic conditioning should be treated as separate questions.  In the scaled Cauchy-trace geometry, Fourier--Bessel modes are orthogonal with explicit weights and the same weights generate the spectrum of equispaced plane-wave traces.  This identifies two different numerical phenomena: small Fourier--Bessel amplitudes are a scaling effect removable by equilibration, whereas small plane-wave Gram eigenvalues indicate a genuine loss of effective trace dimension.  The hybrid identity \eqref{eq:T3} then expresses the best mixed PW--FB approximation as a weighted exponential-fitting problem.

Complex angles require no change in this structure.  A propagating or evanescent plane wave generates the same exponential modal sequence, with evanescence encoded by the modulus of its ESPRIT node.  The exact-recovery and perturbation results in \cref{thm:esprit-exact,thm:esprit-stability}, together with the boundary-trace Lipschitz estimate of \cref{lem:pw-direction-lipschitz}, yield a direct perturbation-to-PWDG estimate under the standard quasi-optimality hypothesis.  After the local space has been fixed, trace-Riesz coordinates and graph--Riesz coordinates act at different levels; \cref{prop:GR} shows that the latter produces the normal form $S-\ii I$.

The computations are consistent with these distinctions.  Sparse ray fields are recovered to roundoff, the plane-wave-only high-order experiment stalls when the resolvable local trace rank saturates, and mixed PW--FB spaces pass that floor without degrading the graph-normalized solve.  The transmission test crosses from propagation to evanescence without a critical-angle switch in the algorithm.  The disk-based modal identities require homogeneous continuation to a containing circle; when this is unavailable, as at the re-entrant corner of the L-shaped test, the same candidate spaces can instead be compared on the physical element boundary.

The practical limits are equally explicit.  A PW rank is attempted only when a sufficiently long modal window is resolved above the data and arithmetic floor; $10^{-12}$ is the binary64 threshold used in the reported calculations, not a universal constant.  Strong evanescence enlarges the perturbation constants exponentially through $e^{\kappa h\sinh\eta_{\max}}$, and broad high-frequency angular spectra reduce the advantage of sparse direction recovery.  Finally, the perturbation-to-DG estimate inherits the stability assumptions of the underlying PWDG problem and does not regularize a near-resonant boundary-value problem.  Extending the analysis to variable coefficients, three-dimensional Trefftz families and fully adaptive meshes requires additional approximation and stability arguments and is left for future work.

\section*{Acknowledgements}
Generative AI tools were used during manuscript preparation for language editing, typesetting assistance and code assistance. The author verified the mathematical derivations, numerical results and final manuscript content.

\section*{Funding}
This research received no external funding.

\section*{Conflict of interest}
The author declares no competing interests.

\section*{Data and code availability}
Source code for the numerical experiments is available to the editors and referees for the purpose of peer review.

\end{document}